\documentclass{amsproc}

 \usepackage{amssymb}

 \usepackage{graphicx}

\usepackage{psfrag}

\newtheorem{theorem}{Theorem}[section]
\newtheorem{lemma}[theorem]{Lemma}

 \newtheorem{prop}[theorem]{Proposition}

\theoremstyle{definition}
\newtheorem{definition}[theorem]{Definition}

\theoremstyle{remark}

\numberwithin{equation}{section}

\def\Z{\mathbb{Z}}
\def\R{\mathbb{R}}
\def\Q{\mathbb{Q}}
\def\C{\mathbb{C}}

\def\calm{\mathcal{M}}
\def\calg{\mathcal{G}}

\def\cale{\mathcal{E}}
\def\cala{\mathcal{A}}

\def\calh{\mathcal{H}}

\newcommand\Hom{\operatorname{Hom}}

\newcommand{\Pf}{\operatorname{Pf}}
\newcommand{\Tr}{\operatorname{Tr}}
\newcommand{\Ind}{\operatorname{Ind}} 
 \newcommand{\spa}{\operatorname{span}} 
\newcommand{\lk}{\operatorname{lk}}
\newcommand{\Sign}{\operatorname{Sign}}

\newcommand\cs{\operatorname{cs}}

\begin{document}

 
\title[Chern-Simons invariants,    instantons,  and  $\Z/2$-homology cobordism]{Chern-Simons invariants, SO(3) instantons,  and  $\Z/2$-homology cobordism}



\author{Matthew Hedden}

 \address{Matthew Hedden: Department of Mathematics, Michigan State University,  
 East Lansing, MI 48824}

\email{mhedden@math.msu.edu}
 
\thanks{This work was supported in part by the National Science Foundation under grants   0604310, 0706979,  0906258, and  1007196. }

\author{Paul Kirk} 
\address{Paul Kirk: Department of Mathematics, Indiana University, Bloomington, IN 47405}
\email{pkirk@indiana.edu}



\date{\today}

\begin{abstract}  We review the $SO(3)$ instanton gauge theory of Fintushel and Stern and recast it in the context of 4-manifolds with cylindrical ends.     Applications to the $\Z/2$ homology cobordism group of $\Z/2$ homology 3-spheres are given. \end{abstract}
\maketitle


\setcounter{tocdepth}{3}
\tableofcontents

\section{Introduction}

Among the  many results of instanton gauge theory, one, due to Furuta,  stands out because no alternative proof has been found despite enormous progress in gauge theory made through the Seiberg-Witten and Ozsv\'ath-Szab\'o theories.  

\medskip

\noindent {\bf Theorem} ({\bf Furuta} \cite{Furuta}, see also \cite{FS}){\bf .}  {\em The Seifert fibered integer homology 3-spheres  $\Sigma(p,q,pqk-1), k=1,2,\cdots\infty$  are linearly independent in the homology  cobordism group, $\Theta^3_H$.}
\medskip

Furuta proved this theorem using the machinery of instantons on pseudofree orbifolds, as developed by Fintushel and Stern \cite{FSpseudo,FS}.   In light of the (perhaps imprecise) expectation that the geometric information contained in Donaldson theory should coincide with that of the more recent approaches to gauge theory, the difficulty in proving this theorem in the context of Seiberg-Witten or Ozsv\'ath-Szab\'o theory is quite mysterious, and motivates the present work.

The aim of this article is to revisit the technique used in the proof of  Furuta's result   and to recast it in the light of advances in the theory of instantons on manifolds with cylindrical ends.  It is our hope that understanding the result in this context will increase its power as a tool for studying smooth cobordism or knot concordance, and will perhaps shed light on how one could extract similar information from the modern invariants.

Before describing some of the nuances involved in the cylindrical end reformulation, we present the following  application afforded by the present approach (see Theorem  \ref{SFQHS}  for a stronger statement and Section \ref{examples} for further applications).

\medskip

\noindent{\bf Theorem 1.} {\em Suppose that $ p,q,d $ are  relatively prime positive odd integers.  For $k$ large enough, the rational homology spheres obtained by $\frac{d}{d^{k}-1}$ surgeries   on the right-handed     $(p,q)$ torus knot    are linearly independent in the $\Z/2$ homology cobordism group $\Theta^3_{\Z/2}$. }  
\medskip


\medskip

 As mentioned, the mechanism underlying  Furuta's argument  relied on the use of orbifold connections and the machinery of instantons on manifolds with cylindrical ends, introduced in Taubes's ground-breaking article \cite{taubes0}. Here, we reformulate the entire    argument in terms of gauge theory on 4-manifolds with cylindrical ends in the manner pioneered by Taubes \cite{taubes0}  and Floer \cite{Floer}.   This yields additional flexibility, in particular providing cobordism obstructions outside of the realm of pseudofree orbifolds.   In  addition, however, our treatment allows us to carefully address   several technical points.      In the present context we are forced to consider the $L^2$ moduli space of instantons when addressing the issue of compactness,  and hence must appeal to the results of \cite{MMR}.  Here, the Chern-Simons invariants  of flat connections provide a lower bound on the quanta of energy that can escape in a non-convergent  sequence of instantons.

 It is perhaps implicit (if not explicitly stated) in \cite{Furuta} and \cite{FS} that the boundary 3-manifolds should not admit  degenerate flat $SO(3)$ connections, and in any case this is automatic when the boundary manifolds are lens spaces or Seifert fibered homology spheres $\Sigma(p,q,r)$ with {\em three} singular fibers.  But for more general 3-manifolds extra care with hypotheses is needed, and we hope our exposition provides this.
In addition, new wrinkles arise in   the enumeration of reducible instantons (see Theorem \ref{CofL}).   These and other points are presented in a context that eschews orbifolds in favor of manifolds with cylindrical ends.    This material  can also be considered as a generalization of  the results of Ruberman and Mati\'c, \cite{matic, ruberman}, which correspond to the case $p_1(E,\alpha)=0$.

In a companion article \cite{HK1}, we will use this machinery to show that a certain infinite family of  untwisted, positive-clasped Whitehead doubles of torus knots is linearly independent in the smooth concordance group.
We anticipate many more applications to the study of concordance and cobordism groups.

Finally, we remark that the general argument is  reminiscent of techniques that use Casson-Gordon invariants to establish linear independence   in the {\em topological} knot concordance group  \cite{jiang}, and this intriguing similarity provides motivation for  further study.   Roughly speaking, in the context of Casson-Gordon invariants arguments focus on one prime $p$ at a time and corresponding $D_{2p}$ (dihedral) representations.  Knots whose branched covers 
have no $p$-torsion in their homology contribute nothing to the obstruction to linear independence. In the Furuta/Fintushel-Stern argument, one focuses on the smallest Chern-Simons invariant of flat $SO(3)$ connections.  Homology spheres whose smallest Chern-Simons invariant exceed  this minimum contribute nothing to the obstruction to linear independence.   

\medskip

\noindent{\bf  Acknowledgments:} It is our pleasure to thank Tom Mrowka, Charles Livingston, and Ron Fintushel for many interesting conversations.  The second author thanks the organizers of the inspiring {\em Chern-Simons Gauge Theory: 20 years after} conference.

  \section{Adapted bundles and instantons }

\subsection{Adapted bundles} We begin by reviewing Donaldson's description \cite{donald1} of gauge theory on adapted bundles over 4-manifolds with cylindrical ends, in the context of $SO(3)$ bundles. Given a flat $SO(3)$ connection $\alpha$ on a 3-manifold $Y$, denote by $H_\alpha^*(Y)$ the cohomology of $Y$ with coefficients in the corresponding flat $\R^3$ bundle.  The flat connection $\alpha$ is called {\em non-degenerate} if $H^1_\alpha(Y)=0$.   If $Y$ is a rational homology sphere, then   the trivial connection  is non-degenerate.   
 \medskip
 
 Consider   a compact oriented 4-manifold $X$  with boundary $\partial X=Y=\sqcup_{i=1}^c Y_i$, a disjoint union of closed 3-manifolds.  Endow $X$ with a Riemannian metric which is isometric to $[-1,0]\times Y $ in a collar neighborhood of the boundary.   Throughout the rest of this section, we make the following assumption.
  
  \medskip
  
 \begin{center}
{\it The 4-manifold $X$ is path connected, satisfies $H^1(X;\Z/2)=0$, and has non-empty boundary $Y=\sqcup_{i=1}^c Y_i$  which is a disjoint union of rational homology 3-spheres.}
\end{center}
 
 \medskip

  Form the non-compact manifold $X_\infty=X\cup_Y \big([0,\infty) \times Y\big)$ by adding an infinitely long collar to each boundary component.  Choose a Riemannian metric on $X_\infty$ whose restriction to each cylinder is the product metric.  The submanifold $ [0,\infty)\times Y_i$ will be called an {\em end} of $X_\infty$.

\begin{definition} An {\em adapted bundle  $(E,\alpha)$ over $X$} is an $SO(3)$ vector bundle $E\to X_\infty$, together with a fixed flat connection $\alpha_i$ on each end $[0,\infty)\times Y_i$.    We use $\alpha$ as  shorthand for the set $\{\alpha_i\}$.  Two  adapted bundles $(E,\alpha)$ and $(E',\alpha')$ are called {\em equivalent} if there is a bundle isomorphism from $E$ to $E'$  identifying  the flat connections $\alpha$ and $\alpha'$ over $[r,\infty)\times Y$ for some $r\ge 0$. 
\end{definition}
 
Any adapted bundle $(E,\alpha)$ is equivalent to one in which  the flat connection $\alpha_i$ on each end  $[0,\infty)\times Y_i$ is in {\em cylindrical form}, that is $\alpha_i=\pi^*(\tilde{\alpha}_i)$ where $\tilde{\alpha}_i$ is a flat connection on $Y_i$ and $\pi:[0,\infty)\times Y_i\to Y_i$ is the projection to the second factor.  Indeed, since $\alpha$ is flat,   there always exists a gauge transformation $g: E\to E$ which equals the identity on the interior of $X$  so that $g^*(\alpha)$ has this form, obtained by parallel transport along rays $[0,\infty)\times  \{y\}$. We will tacitly assume that each $\alpha_i$ is in cylindrical form whenever convenient, and use the same notation $\alpha_i$ for the restriction of this flat connection to $Y_i$. 

  Since $X$ has non-empty boundary, $SO(3)$ vector bundles over $X_\infty$ are isomorphic if and only if their second Stiefel-Whitney classes are equal.    In particular,  if 
 $(E,\alpha)$ and $(E',\alpha')$ are equivalent  adapted bundles over $X_\infty$, then $w_2(E)=w_2(E')$.

 \medskip

 \subsection{The Pontryagin charge, instantons,  and Chern-Simons invariants}\label{chsi}
Given any  $SO(3)$ connection $A$  on a bundle $E$ over a (not necessarily compact) 4--manifold, $Z$, let $F(A)$ denote its curvature 2-form.  Define the {\em Pontryagin charge of} $A$  to be the real number
\begin{equation}\label{csp1}
p_1(A)=-\frac{1}{8\pi^2}\int_{Z}  \Tr(F(A)\wedge F(A)) ,
\end{equation}
provided this integral converges.  When $Z$ is closed, $p_1(A)=\langle p_1(E),[Z]\rangle\in \Z$.
\begin{definition} An $SO(3)$  connection $A$ on a  Riemannian manifold $Z$ is called an {\em instanton} if its curvature form satisfies the equation $$F(A)=-*F(A).$$   \end{definition}
If $A$ is an instanton, then 
$$p_1(A)=\frac{1}{8\pi^2}\int_{Z}  \Tr(F(A)\wedge *F(A))=
\frac{1}{8\pi^2}\| F(A)\|^2_{L^2}$$
whence $p_1(A)$ is defined and non-negative for any instanton $A$ whose curvature has finite $L^2$ norm.

Given an adapted bundle $(E,\alpha)$ over $X_\infty$, choose an $SO(3)$ connection $A_0$ on $E$ which extends   the given flat connection $\alpha$ on the union of the ends. Then its curvature $F(A_0)$ is supported on the compact  submanifold $X$,  and so $p_1(A_0)\in \R$ is defined.   Given any other extension $A_1$ of $\alpha$, then Chern-Weil theory shows that $ \Tr(F(A_1)\wedge F(A_1))=  \Tr(F(A_0)\wedge F(A_0)) + d\gamma$  where
$$\gamma=- \int_0^1 \Tr\big((A_0-A_1)\wedge 2F(tA_0 +(1-t)A_1)\big)dt.   $$  
See, for example, \cite[Lemma 3.3]{Wells}.  Since $A_0-A_1$ vanishes on the collar, Stokes' theorem implies that  $p_1(A_0)=p_1(A_1)$, so we denote this quantity by $p_1(E,\alpha)$.   

\medskip

\begin{definition} \label{csdefn} Given a pair $\alpha_0, \alpha_1$ of  $SO(3)$ connections on the same bundle $E$ over 
closed 3-manifold $Y$, define the  {\em relative Chern-Simons invariant},
$\cs(Y,\alpha_0,\alpha_1)\in \R$,  by $$\cs(Y,\alpha_0,\alpha_1)=p_1(\R\times E,\alpha),$$   where $\alpha$ equals $\alpha_0$ on $(-\infty, 0]\times Y$ and $\alpha_1$ on $[1,\infty )\times Y$.

  \end{definition}

 \begin{definition} Given a   connection $\alpha$ on an $SO(3)$ bundle  $E $ over a closed 3-manifold $Y$, the bundle    extend  to a bundle   with connection $A$ over some 4-manifold $Z$ with boundary $Y$, since $H_3(BSO(3))=0$.    Define the {\em Chern-Simons invariant  of $\alpha$  modulo $\Z$},  $\cs(Y,\alpha)\in \R/\Z$ by 
 $$\cs(Y,\alpha)=p_1(A) \text{ mod } \Z.$$ 
Since the integral of the first Pontryagin form of a connection   over a {\em closed} 4-manifold  is an integer, it follows that    $\cs(Y,\alpha)$ is well defined in $ \R/\Z$.\end{definition}

These two invariants are related by   
  $$\cs(Y,\alpha_0,\alpha_1)= \cs(Y,\alpha_1)- \cs(Y,\alpha_0)\text{ mod }\Z $$
  whenever the left side is defined.
  
    If $\alpha$ extends to a {\em flat} connection $A$ over some 4-manifold then $p_1(A)=0$, and  hence   the Chern-Simons invariant  modulo $\Z$ is  a flat cobordism invariant.  
 In particular, if $E\to Y$ is the   trivial bundle and $\alpha$ a flat connection on $E$, then   $\cs(Y,\alpha)$ has a canonical lift to $\R $, namely $\cs(Y,\theta,\alpha)$, where $\theta$ denotes the trivial connection with respect to the given trivialization of $E$.   
 
  Bundle automorphisms $h:E\to E$  covering the identity which preserve the $SO(3)$ structure group are called {\em gauge transformations}. The group $\calg(E)$ of gauge transformations acts on the space of  $SO(3)$ connections $\cala(E)$ by pullback.

   If $g:E\to E$ is a  gauge transformation,  then $\cs(Y, \alpha, g^*(\alpha))$ equals $p_1(\tilde E)$, where $\tilde E$ is the bundle over $  Y \times S^1$ obtained by taking the mapping torus of $g$. Since $Y \times S^1$ is closed, Chern-Weil theory implies that $p_1(\tilde E)$ is an integer, and so we conclude that $\cs(Y,\alpha, g^*(\alpha))\in \Z$.  Notice that if $g$ is homotopic to the identity gauge transformation, then the bundle $\tilde E$  is isomophic to the pullback of $E$ over $Y$ via the projection $\pi:Y\times S^1\to Y$, and hence  $\cs(Y,\alpha, g^*(\alpha))=p_1(\pi^*(E))=0$. A similar argument shows that  the integer $\cs(Y,\alpha, g^*(\alpha))$ depends only on the path component of $g$ in  $\calg (E)$.  It follows easily that  the reduction modulo $\Z$ of $\cs(Y,\alpha_0,\alpha_1)$ depends only on the orbits of $\alpha_0$ and $\alpha_1$  under the action of $\calg(E)$. Similarly, $\cs(Y,\alpha)\in \R/\Z$ depends only on the orbit of $\alpha$.

Dold and Whitney proved that $SO(3)$ bundles over 4-complexes  are determined up to isomorphism by    
their second Stiefel-Whitney class $w_2$ and their first Pontryagin class $p_1$ \cite{DW}.  Moreover, these are related by  $p_1\equiv \mathcal{P} (w_2) $ mod $4$, where $ \mathcal{P}:H^2(-;\Z/2)\to H^4(-;\Z/4)$ denotes the Pontryagin square, a cohomology operation which lifts the cup square $   H^2(-;\Z/2)\to H^4(-;\Z/2)$, $y\mapsto y^2$.

As a consequence, if $Y$ is a $\Z/2$ homology 3-sphere, then  $H^2(Y\times S^1;\Z/2)=0$, and so $\cs(Y, \alpha, g^*(\alpha))\in 4\Z$ for any gauge transformation $g$.

 \subsection{The moduli space}
    
Let $(E,\alpha)$ be an adapted bundle over $X_\infty$.   The elliptic operators $*d_{\alpha_i}-d_{\alpha_i}*$ acting on even degree $E|_{Y_i}$--valued forms 
 on $Y_i$ are self-adjoint with discrete spectrum.  Choose (and fix for the remainder of this section) a  $\delta>0$ smaller than the absolute value of any non-zero eigenvalue of  any of the $*d_{\alpha_i}-d_{\alpha_i}*$.   Let $f:X_\infty\to [0,\infty)$ be a smooth function which equals 0 on $X$ and ${t\delta}$ on $\{t\} \times Y_i$ for $t>1$. Then   define  the weighted Sobolev spaces $L^{p,\delta}_n$ of sections of bundles over $X_\infty$ as the completion of the space of compactly supported sections with respect to the norm $\| \phi\|_{L^{p,\delta}_n}:=\| e^{f}\phi\|_{L^p_n}$.

Fix an extension $A_0$ of $\alpha$ to $E$.      Identify $E$ with $so(E)$  via the identification $\R^3=so(3)$ of $SO(3)$ representations, and let $SO(E)$ denote the $SO(3)$ bundle automorphisms. Connections on $E$  take the form $A_0+a$ for a 1-form $a$ with values in $E$. Following   \cite[Chapter 4]{donald1}, define 
 $$ \cala^\delta(E,\alpha)=\{A_0+a \  |  \  a\in L^{3,\delta}_1(\Omega^1_{X_\infty}(E))\} 
$$ and $$
 \calg^\delta(E,\alpha)=
 \{ g \in SO(E)\ | \  \nabla_{A_0}g\in L^{3,\delta}_1(T^*X_\infty\otimes \text{End}(E)) \}$$
(where $\nabla_{A_0}$ denotes the covariant derivative associated to $A_0$). Then pulling back connections extends to an action of  the completions $\calg^\delta(E,\alpha)$  on $\cala^\delta(E,\alpha)$.  
 The {\em moduli space $ \calm(E,\alpha) $ of instantons on $(E,\alpha)$} is defined to  be
 $$\calm(E,\alpha)=\{A\in \cala^\delta(E,\alpha)\ | \ F(A)=-*F(A)\}/\calg^\delta(E,\alpha).$$

The idea underlying these definitions is that $\cala^\delta(E,\alpha)$ consists of connections that limit    exponentially along the ends of $X_\infty$ to $\alpha$, and $\calg^\delta(E,\alpha)$ consists of gauge transformations that limit exponentially along the ends of $X_\infty$ to  a  gauge transformation   that stabilizes $\alpha$.

 \begin{definition} A {\em reducible connection} on $(E,\alpha)$ is a connection $A$ in $\cala^\delta(E,\alpha)$ whose stabilizer $\Gamma_A\subset\calg^\delta$ is non-trivial. Let $\calm(E,\alpha)_{red}\subset  \calm(E,\alpha)$ denote the subset of gauge equivalence classes of reducible instantons. Its complement
 $\calm(E,\alpha)\setminus\calm(E,\alpha)_{red}$  is denoted
 $\calm(E,\alpha)^*$.   \end{definition}

Restricting a gauge transformation to the fiber of $E$ over a base point in $X$ embeds the stabilizer $\Gamma_A$ in $SO(3)$ and identifies it with the centralizer of the holonomy group of $A$ at the base point. In particular, the subgroup $\calg_0^\delta(E,\alpha)\subset \calg^\delta(E,\alpha)$  of gauge transformations which restrict to the identity in the fiber of $E$ over the base  point of $X$  acts freely on $\cala^\delta(E,\alpha)$, and the quotient $SO(3)\cong 
 \calg^\delta(E,\alpha)/ \calg_0^\delta(E,\alpha)$ acts on
  $\cala^\delta(E,\alpha)/\calg_0^\delta(E,\alpha)$ with stabilizers $\Gamma_A$.  The restriction of this action to the irreducible instantons defines a principal $SO(3)$ bundle $\cale\to \calm(E,\alpha)^*$ called the {\em base point fibration} \cite[Section 9]{FSpseudo}.

\subsection{Calculation of the index}

 Define $\Ind^+(E,\alpha)$ to be the Fredholm index of the operator 
 $$S^\delta_{A_0}=d^*_{A_0}\oplus d^+_{A_0}:L^{3,\delta}_1(\Omega^1_{X_\infty}(E))\to 
L^{3,\delta} (\Omega^0_{X_\infty}(E)\oplus \Omega^+_{X_\infty}(E))$$
where $\Omega^+$ denotes the self-dual 2-forms: $\omega=*\omega$.
The Atiyah-Patodi-Singer theorem can be used to calculate  this index.  
Let $\eta(Y,\alpha)$ denote the Atiyah-Patodi-Singer  spectral invariant  $\eta(B_{\alpha},0)$ of  the   operator on $Y$:
\begin{equation*}
\begin{split} &B_{\alpha}:\Omega^0_{Y}(E|_Y)\oplus \Omega^1_{Y}(E|_Y) \to
\Omega^0_{Y}(E|_Y)\oplus \Omega^1_{Y}(E|_Y), \\  &B_{\alpha}(\phi_0,\phi_1) =
(- d_{\alpha}^*\phi_1,  d_{\alpha}\phi_0+*d_{\alpha}\phi_1), \end{split}
\end{equation*}
and let $$\rho(Y,\alpha)=\eta(Y,\alpha)-3\eta(Y),$$ denote the Atiyah-Patodi-Singer $\rho$ invariant of $(Y,\alpha)$. (In general,  for a connection  $\gamma $ on a  bundle $V\to Y$, $\rho(Y,\gamma)=\eta(Y,\gamma)-\dim(V)\cdot \eta(Y,\theta)$, where $\theta$ is the trivial connection on the trivial $1$-dimensional line bundle.) The real number $\rho(Y,\alpha)$ depends only on the gauge equivalence class of $\alpha$.

In general, the kernel of $B_{\alpha_i}$ equals the direct sum $\calh^0_{\alpha_i}(Y_i)\oplus \calh^1_{\alpha_i}(Y_i)$ of the $d_{\alpha_i}$-harmonic $0$- and $1$-forms on $Y_i$ with coefficients in the flat bundle supporting $\alpha_i$ (see  \cite[Section 2.5.4]{donald1}). The assumption that  each flat connection $\alpha_i$ is non-degenerate and the Hodge theorem imply that  $\ker B_\alpha\cong\oplus_{i=1}^c  H^0_{\alpha_i}(Y_i)$.
Setting  $h_\alpha= \dim H^0_{\alpha}(Y)$  and $h_{\alpha_i}=\dim H^0_{\alpha}(Y_i) $ we conclude that  $h_\alpha =\sum h_{\alpha_i}$.   Note that $\rho(Y_i,\alpha_i)=0$ and $h_{\alpha_i}=3$  when $\alpha_i$ has trivial holonomy,   as such a connection is gauge equivalent to the trivial connection.

The discussion following the proof of Proposition 3.15 of  \cite{APS1} shows that the index $\Ind^+(E,\alpha)$ is equal to  the Atiyah-Patodi-Singer index  $\Ind S^{APS}_{A_0} $ of the corresponding operator on (the compact  manifold) $X$ with Atiyah-Patodi-Singer boundary conditions.  This uses the non-degeneracy of the $\alpha_i$ to conclude that $L^2$ solutions of $S_{A_0}\phi=0$ decay along the ends faster than $e^{-t\delta}$.  Thus $\Ind^+(E,\alpha)$ can be computed using the Atiyah-Patodi-Singer theorem.  For the convenience of the reader we outline the  calculation; similar calculations can be found in the literature; for example  it is a consequence of (the more general) \cite[Proposition 8.4.1]{MMR}.    Let $b^+(X)$ denote the dimension of a maximal positive definite subspace of $H^2(X;\R)$ with respect to the intersection form (see Section \ref{reduciblesect} below for the definition of the intersection form in this context).

\begin{prop}\label{APSindex} Let $(E,\alpha)$ be an adapted bundle over a path connected 4-manifold $X$ with $c\ge 1$ boundary components.  Assume further that the flat connection $\alpha_i$ on each   $Y_i$ is non-degenerate, and that $H^1(X;\Q)=0=H^1(Y;\Q)$.  

Then 
 \begin{equation*}\Ind^+(E,\alpha)= 2p_1(E,\alpha) - 3(1+b^+(X))   +\tfrac{1}{2}\sum_{\alpha_i{\text{\tiny \rm nontrivial}}}(3-h_{\alpha_i}-\rho(Y_i,\alpha_i)).
\end{equation*}
 
\end{prop}

\begin{proof}  We use the Atiyah-Patodi-Singer index theorem \cite{APS1} twice, once on the bundle $E$ with operator 
$$S_{A_0}^{APS}=d_{A_0}^*\oplus d^+_{A_0}:L^2_1(\Omega^1_{X}(E);P_\alpha^+)\to L^2(\Omega^0_X(E)\oplus \Omega^+_X(E))$$
with index 
$$\Ind S^{APS}_{A_0}=\int_{X}f_{X,E}(R_X,F(A_0))-\tfrac{1}{2}(h_\alpha + \eta(Y,\alpha)),$$
 and again on the trivial 1-dimensional $\R$ bundle with trivial connection  
 $$S^{APS}=d^*\oplus d^+:L^2_1(\Omega^1_{X};P^+)\to L^2(\Omega^0_X\oplus \Omega^+_X)$$
with index 
$$\Ind S^{APS}=\int_{X} f_{X,\epsilon}(R_{X},\theta)-\tfrac{1}{2}(h + \eta(Y)).$$

In these formulas $P^+_\alpha$ denotes the spectral projection onto the non-negative eigenspace of $B_\alpha$, and $\eta(Y_i,\alpha_i)$ and $h_\alpha$, are as above.  The corresponding unadorned symbols refer to their analogues on the trivial, untwisted $\R$ bundle $\epsilon$.  Since   $Y_i$ is a rational homology sphere,    $h =c$.

The  (inhomogeneous) differential form $f_{X,E}( R_{X},F(A))$ is a function of the Riemannian curvature $R_X$ and the curvature $F(A)$ of a connection $A$.  
It equals 
$$\big(\dim E -\tfrac{1}{8\pi^2} \Tr(F(A)\wedge F(A))\big)\big(2+  \tfrac{1}{2}(\chi(R_{X})+\sigma(R_{X}))   \big),$$  where  $\chi(R_{X})$ is the Euler form and $\sigma(R_{X})$ is the Hirzebruch signature form.  Hence
$$\int_{X}f_{X,E}( R_{X},F(A))=  
2p_1(A_0)+3\int_{X}f_{X,\epsilon}( R_{X},\theta).
 $$
By convention, the integral of an inhomogeneous form over an $n$-manifold is defined to be the integral of its $n$-dimensional component.  

Since $\Ind^+(E,\alpha)$ is equal to    $\Ind S^{APS}_{A_0} $,   
 subtracting $3\Ind S^{APS}$ from $\Ind S^{APS}_{A_0}$ and simplifying yields
 \begin{equation}
\label{2.15.1}
\Ind^+(E,\alpha)= 2p_1(E,\alpha) + 3\Ind S^{APS} -\tfrac{1}{2}\rho(Y,\alpha) +\tfrac{1}{2}(3c-h_\alpha).
\end{equation}

To obtain the desired formula, we must show $\Ind S^{APS}=-1-b^+(X)$.   The kernel of $S^{APS}$ consists of $L^2$ harmonic  $1$-forms on $X_\infty$. Proposition 4.9 of \cite{APS1} identifies this space with the image $H^1(X,Y;\R)\to H^1(X;\R)$.  The hypothesis that $H^1(X;\Q)=0$  (which holds, in any event, by our standing assumption that $H^1(X;\Z/2)=0$) implies that $H^1(X;\R)=0$, and so $\ker S^{APS}=0$.
 
We now identify the cokernel of $S^{APS}$.  The cokernel is isomorphic to the kernel of the adjoint which, in turn, is isomorphic to the extended $L^2$ solutions to $S^*\phi=0$, i.e., the space of pairs $\phi=(\phi_0,\phi_+)\in \Omega_{X_\infty}^0\oplus \Omega_{X_\infty}^+$ satisfying $d\phi_0+d^*\phi_+=0$ and for which $(\phi_0,\phi_+)|_{ \{t\}\times Y}-\gamma$  decays exponentially on the collar  for some harmonic form $\gamma=(\gamma_0, \gamma_1) \in \ker B$ on $Y$ \cite{APS1}.  Since $Y$ is a rational homology sphere, $\ker B\cong H^0(Y)$.  It follows that $\gamma=(\gamma_0,0)$ where $\gamma_0$  is a harmonic 0-form, that is, a locally constant function on $Y$. Thus $\phi_0|_{ \{t\}\times Y}-\gamma_0$  and $\phi_+|_{ \{t\}\times Y}$  are exponentially decaying forms.     

Let $X(t)= X\cup ([0,t]\times Y)$. Then  
$$\langle d\phi_0, d^*\phi_+\rangle_{L^2(X(t))}=\pm \langle d\phi_0|_{ \{t\}\times Y},\phi_+|_{ \{t\}\times Y}\rangle_{L^2({ \{t\}\times Y})} \xrightarrow{t\to\infty}0$$
 This implies that  $d\phi_0$ and $d^*\phi_+$ are $L^2$-orthogonal. Therefore,  
 $d\phi_0+d^*\phi_+=0$ implies that   $\phi_0$ is a harmonic 0--form (that is, a constant function) and $\phi_+$ is an exponentially decaying harmonic self-dual 2-form.   Since $X$ is connected we conclude coker$S^{APS}=1+b^+(X)$, and hence 
  $\Ind S^{APS}=-1-b^+(X).$  
  
  If $\alpha_i$ has trivial holonomy, then $\rho(Y,\alpha_i)=0$ and $h_{\alpha_i}=\dim H^0_{\alpha_i}(Y)=3$.  Therefore, 
  \begin{equation}
\label{2.15.2}-\rho(Y,\alpha)+3c-h_\alpha=\sum_{i}3-h_{\alpha_i} -\rho(Y_i,\alpha_i)=
  \sum_{\alpha_i{\text{\tiny \rm nontrivial}}}3-h_{\alpha_i}-\rho(Y_i,\alpha_i)
\end{equation}
Combining Equations (\ref{2.15.1}) and (\ref{2.15.2})  yields the result.
\end{proof}


Suppose that  one of the flat connections, say $\alpha_1$, on the boundary component $Y_1$ is trivial with respect to some trivialization of $E$ over $Y_1$. Suppose further that  $Z$ is a compact, negative definite 4--manifold with boundary $-Y_1\sqcup_jW_j$, with the trivialized bundle $Z\times \R^3\to Z$ over it. Then glue $X$ to $Z$ along $Y_1$ to produce the manifold $X'=X\cup_{Y_1} Z$.  Let $E'$  be the bundle over $X'$  obtained by gluing $E$ to the trivial bundle using the given trivializations over $Y_1$. Extending the connection $A_0$ over $X'$ by the trivial connection over $Z$ gives a new adapted bundle $(E',\alpha')$ over $X_\infty'$.   Since $p_1(E,\alpha)=p_1(E',\alpha')$, Proposition  \ref{APSindex}  then shows that $\Ind^+(E,\alpha)=\Ind^+(E' , \alpha')$.

\medskip
If $\Ind^+(E,\alpha)\ge 0$, $p_1(E,\alpha)>0$, and  each $\alpha_i$ is non-degenerate, then by varying the Riemannian metric inside $X$ if necessary, $\calm(E,\alpha)^*$ is a smooth orientable manifold of dimension $\Ind^+(E,\alpha)$.  In the present context this is a consequence of  Lemma 8.8.4, Theorem  9.0.1 and  Remark 9.2.7  of \cite{MMR}.  
  When consulting \cite{MMR}, it is useful to note that because we assume the flat connections $\alpha_i$ are non-degenerate, the appropriate path components of the $L^2$ moduli spaces and thickened moduli spaces of \cite{MMR}   coincide with $\calm(E,\alpha)$.

When $p_1(E,\alpha)=0$, then $\calm(E,\alpha)$ is the moduli space of  flat connections and hence is unchanged by varying the Riemannian metric on $X_\infty$. In this situation one can instead perturb the 
equation $F(A)+*F(A)=0$ to make $\calm(E,\alpha)^*$  a smooth orientable manifold of dimension $\Ind^+(E,\alpha)$; see \cite{Furuta2,donald1}.

\subsection{Compactness}
A critical question about $\calm(E,\alpha)$ is whether it is compact.  To understand this, we define an invariant $\hat\tau(Y,\alpha)\in(0,4]$, essentially the minimal relative $SO(3)$ Chern-Simons invariant over the  path components of $Y$.  This invariant provides a sufficient condition to guarantee compactness of $\calm(E,\alpha)$ (Proposition \ref{compactness} below).  

We first introduce some notation.  For any path connected space $Z$ and compact Lie group $G$, denote by $\chi(Z,G)$ the space
$$\chi(Z,G)=\Hom(\pi_1(Z), G)/{\text{\small conjugation}}.$$  This   space is a compact real algebraic variety and is analytically isomorphic to the space of gauge equivalence classes of flat $G$ connections 
  \cite{FKK}.  To a flat $G$ connection $\alpha$ we associate its holonomy hol$_\alpha\in \chi(Z,G)$.
 Notice that $\chi(Z,G)$ is partitioned into disjoint compact subspaces corresponding to the isomorphism classes of $G$-bundles over $Z$ which support the various flat connections.
  When $Z$ is not path connected then take $\chi(Z,G)$ to be the product of the $\chi(Z_i,G)$  over the path components $Z_i$ of $Z$.   

We can now define the invariant $\hat\tau(Y,\alpha)\in(0,4]$, which is an adaptation of invariants found in \cite{Furuta, FS} and \cite[Definition 6.3.5]{MMR} to our context.

 For each boundary component $Y_i$ of $X$, the bundle $E|_{Y_i}$ carries the flat connection $\alpha_i$ and  is determined up to isomorphism by the gauge equivalence class of $\alpha_i$.   Hence $\alpha_i$ determines its second Stiefel-Whitney class. Thus  the holonomy gives a decomposition into a disjoint union
 \begin{equation}
 \label{characters}
 \chi(Y_i,SO(3))=\underset{{w\in H^2(Y_i;\Z/2)}}  \sqcup\chi(Y_i,SO(3))_w 
 \end{equation}

  If $\gamma$ is any flat connection on $E|_{Y_i}$ and $g$ any gauge transformation of $E|_{Y_i}$, then $\cs(Y_i,  \alpha_i,  \gamma)-\cs(Y_i,  \alpha_i,  g^*(\gamma))$ is an integer, as explained in Section \ref{chsi}.  This implies that $\cs(Y_i,\alpha_i,-)$ descends to a well-defined $\R/\Z$ valued function on the compact subspace  $\chi(Y_i,SO(3))_{w_2(E|_{Y_i})}\subset \chi(Y_i,SO(3))$. This function is locally constant, see e.g.  \cite{KK1}, and hence takes finitely many values in $\R/\Z$. 
This in turn implies that the set $$
S(\alpha_i)=\{\cs(Y_i,  \alpha_i,  \gamma)\text{ \rm mod } 4\ | \ \gamma\text{ a flat connection on } E|_{Y_i}\}\subset \R/4\Z$$
is finite, since it is contained in the preimage of a finite set under $\R/4\Z\to\R/\Z$.

At the moment, taking Chern-Simons invariants mod 4 might seem unmotivated, but it occurs for two reasons. First, in the compactness result below, energy can only bubble off at interior points in multiples of 4. Second, if $Y_i$ is a $\Z/2$ homology sphere, then $\cs(Y_i,  \alpha_i,  \gamma)-\cs(Y_i,  \alpha_i,  g^*(\gamma))=\cs(Y_i,  g^*(\gamma),  \gamma)$ is in fact always four times an integer, as explained in Section \ref{chsi}. 

  \begin{definition}\label{defoftau}
    Let $b:\R/4\Z\to (0,4]$ be the obvious bijection.  Define $ \tau(Y_i,\alpha_i)\in(0,4]$ to be the minimal value which $b$ attains   on the set $S(\alpha_i)$.  Informally, $\tau(Y_i,\alpha_i)$ is the minimal   relative Chern-Simons invariant  $\cs(Y_i,\alpha_i,\gamma)$ for $\gamma\in 
  \chi(Y_i,SO(3))_{w_2(E|_{Y_i})}$, taken modulo 4.

Then define
\begin{equation}\hat\tau(Y,\alpha)=\min_{Y_i\subset \partial X}\{\tau(Y_i, \alpha_i)\}\in(0,4]
\end{equation}
\end{definition}

Note that $\hat\tau(Y,\alpha)$ depends only on   $Y=\sqcup_iY_i$   and the gauge equivalence classes of the flat connections $\alpha_i$.

\medskip

The following lemma will prove useful for estimating $\hat{\tau}(Y,\alpha)$.  We omit the  proof, which is a simple consequence of the remarks in the paragraph following Definition \ref{csdefn}.

 \begin{lemma}\label{mod1mod4} Let $E\to Y$ be an  $SO(3)$ vector bundle over a closed 3-manifold and $\alpha, \gamma$ two flat connections on $E$. 
Choose any pair $(W_\alpha, E_\alpha)$ where $W_\alpha$ is a 4-manifold  with boundary $Y$,  $E_\alpha\to W_\alpha$ is a vector bundle extending $E$.  Similarly choose $(W_\gamma, E_\gamma)$.  

Then $\cs(Y,\alpha, \gamma)$ mod 4, taken in $(0,4]$, is greater than or equal to the fractional part of 
$p_1(E_\gamma,\gamma)-p_1(E_\alpha,\alpha)$, taken in $(0,1]$.
In particular, if $ p_1(E_\gamma,\gamma) $ is   rational   with denominator dividing $k\in \Z_{>0}$ for all flat connections $\gamma$ on $E$, then $\tau(Y,\alpha)\ge \frac{1}{k}$ for all $\alpha$. \qed
\end{lemma}
 
 As mentioned, $\hat\tau(Y,\alpha)$ provides a sufficient condition for compactness of the moduli space.

\begin{prop} \label{compactness} Suppose that $(E,\alpha)$ is an adapted bundle with  each $\alpha_i$ non-degenerate   and $0\leq p_1(E,\alpha)<\hat\tau(Y,\alpha)$.
 Then $\calm(E,\alpha)$ is compact.

\end{prop}

 \begin{proof}  This follows from the {\em convergence with no loss of energy} theorem of Morgan-Mrowka-Ruberman \cite[Theorem 6.3.3]{MMR} which says that a sequence of gauge equivalence classes of finite energy  instantons on $(E,\alpha)$ have a geometric limit which is the union of an idealized instanton on $X_\infty$ and a union of finite energy instantons on tubes (see also   \cite{taubes0} and \cite{Floer} for earlier versions of this result).  We give an outline and refer the reader to \cite{MMR} for details.  
 
 This  result states that any sequence of gauge equivalence classes of instantons $\{[A_i]\}_{i=1}^\infty$ in $\calm(E,\alpha)$  has a subsequence which converges in the following sense. 
The limit is a sequence of connections $(A_\infty, B_1,\cdots, B_k)$, where  $A_\infty$  is an idealized instanton on an adapted bundle $(E',\alpha')$ over $X_\infty$, and  $B_j ,  \  j=1,\cdots, k$ are instantons on the bundle $\R\times E|_{Y_{i_j}}$ over $\R\times Y_{i_j}$.  Here $i_j\in \{1,2,\cdots, c\}$.

More precisely, $A_\infty$ is a smooth  instanton
away from a finite set of points $\{x_m\}$ in $X_\infty$, and the curvature density $\|F(A_\infty)\|^2$ is the sum of a non-negative smooth function and a Dirac measure supported on these points with weight $32\pi^2 n_m$ for some non-negative integers $n_m$.
The instantons $B_j$ on the cylinders have finite,  positive energy: $0<\|F(B_j)\|^2<\infty$.  In particular none of the $B_j$ are flat.

 The connections $A_\infty$ and $B_j$  have  compatible boundary values, which means that   the formal sums of gauge equivalence classes of flat connections on $Y=\partial X$ are equal   
  $$ \alpha'_1 +\cdots+\alpha'_c=\alpha_1+\cdots +\alpha_c  +\sum_{j=1}^k \big(\lim_{r\to -\infty}B_j(r) -
  \lim_{r\to  \infty}B_j(r)\big),$$
  where $B_j(r)=B_j|_{\{r\}\times Y_{i_j}}$.
 Finally, no energy is lost, i.e. there is an equality of non-negative numbers
 \begin{equation}\label{energy}
p_1(E,\alpha)=p_1(E',\alpha')+\sum_{\{x_m\}} 4n_m+ \sum_{j=1}^k p_1(B_j).
\end{equation}

  Since $p_1(E,\alpha)<4$ and every term in (\ref{energy}) is non-negative,   the idealized instanton $A_\infty$ is necessarily an (honest) instanton, that is, each $n_m$ is zero, the set $\{x_m\}$ is empty, the bundle $E'$ equals $E$, and  $A_\infty$ is a smooth instanton on  $E$.    Informally, no energy can bubble off 
  at interior points.    
  
  Next, suppose that $k> 0$. Then the identification of the limiting flat connections implies that  one of the instantons $B_j$   
   on $\R\times E|_{Y_{i_j}}$ over $\R\times Y_{i_j}$ has left handed limit one of the $\alpha_i$, that is,
   $\lim_{r\to -\infty}B_{j}(r)=\alpha_{i_j}$.  Denote by $\gamma$ the limit $\lim_{r\to  \infty}B_{j}(r)$.  Then 
   $p_1(B_j)\equiv cs(Y_{i_j},\alpha_{i_j}, \gamma)$ mod 4, and since $B_j$ is not flat, $p_1(B_j)>0$.  Therefore
   $$p_1(B_{j})\ge \tau(Y_{i_j},\alpha_{i_j})\ge \hat\tau(Y,\alpha)>p_1(E,\alpha),$$ which is impossible, since every term in (\ref{energy}) is non-negative.  Therefore $k=0$.  Informally, no energy can escape down the ends of $X_\infty$.    Thus $A_\infty\in \calm(E,\alpha)$, as desired. 
 \end{proof}

\subsection{Reducible connections.} \label{reduciblesect}
We now discuss reducible adapted bundles and reducible connections.  We first  
 remind the reader of the extended intersection form of  a 4-manifold whose boundary is a union of rational homology spheres.
Start with the pairing $$H^2(X;\Z)\times H^2(X,Y;\Z)\to \Z,\ \ \ \ \ \ x\cdot y=(x\cup y)\cap[X,Y].$$    Let  $d$ be any positive integer so that $d\cdot H^2(Y;\Z)=0$.  Then this  produces a  well-defined   pairing 
 \begin{equation}
\label{intersectionform}
H^2(X;\Z)\times H^2(X;\Z)\to \tfrac{1}{d}\Z 
\end{equation}
 by  $$ \ x\cdot y:= \tfrac{1}{d}(x\cdot z), \text{ where }z\in H^2(X,Y;\Z)\text{ satisfies } i^*(z)=d\cdot y.$$
  
\noindent This can also be described as the restriction of the composite
 \begin{equation*}\begin{split}
 H^2(X;\Z)\times H^2(X;\Z)\to H^2(X;\R)\times H^2(X;\R)& \cong H^2(X,Y;\R)\times H^2(X,Y;\R)    \\
  &\xrightarrow{\cup}H^4(X,Y;\R)\cong\R.
\end{split} \end{equation*}

  Note that $x\cdot y=0$ if $x$ or $y$ is a torsion class.  Thus we say {\em $X$ has a negative definite intersection form} or more briefly  {\em $X$ is negative definite} if $x\cdot x<0$ whenever $x$ is not a torsion class.  One defines positive definite analogously. Call $X$  {\em indefinite} if it is neither positive nor negative definite.   
  
  Let $b^+(X)$ denote  the maximal dimension of any subspace of $H^2(X;\R)$ on which the intersection form is positive definite, and similarly define $b^-(X)$. The assumption that $X$ has boundary a union of rational homology 3-spheres implies that $\dim(H^2(X;\R))=b^+(X)+b^-(X)$. Therefore  $b^+(X)=0$ if and only if $X$ is negative definite.
  
  If $X$ is indefinite, then for a generic Riemannian metric the moduli space 
  $\calm(E,\alpha)$ is a smooth manifold. However, in contrast to some arguments in gauge theory,    the argument of \cite{FS,Furuta} makes use of the fact that reducibles cannot be perturbed away when $X$ is negative definite.  This technique originates in Donaldson's proof of his celebrated theorem  that   definite intersection forms of closed smooth 4-manifolds are diagonalizable \cite{donaldsonA}.  
Therefore we make the following assumption for the remainder of this section:
\begin{center}{\em The intersection form of $X$ is negative definite; that is, $x\cdot x<0$ unless $x$ is a torsion class. }\end{center} 

 Near the orbit of a reducible instanton $A$, an application of the slice theorem shows that  $\calm(E,\alpha)$ has the structure $V/\Gamma_A$ for a vector space $V$ of dimension $\Ind^+(E,\alpha)+\dim\Gamma_A$ \cite[Theorem 4.13]{donald1}.  This is understood by considering the  base point fibration (described above)
 $ SO(3)\to \cala^\delta(E,\alpha)/\calg_0^\delta(E,\alpha)\to \cala^\delta(E,\alpha)/\calg^\delta(E,\alpha)$. Restricting to $\calm(E,\alpha)$ gives an $SO(3)$ action on a smooth manifold   of dimension $\Ind^+(E,\alpha)+3$ with quotient   $\calm(E,\alpha)$ and stabilizer  $\Gamma_A$ over the instanton $A$.

\medskip

Subgroups of  $SO(3)$ have several possible centralizers, and this can lead to different types of singularities in the instanton moduli space.   For our purposes, it is sufficient to deal solely with orbit types of reducibles that have $SO(2)$-stabilizers.  In light of our assumption that $H_1(X)$ has no 2-torsion, we can ensure this with the following  (which, hereafter, will be assumed unless otherwise stated):
\begin{center}{\em The adapted bundle $(E,\alpha)$ is non-trivial. Equivalently, $\alpha$ does not extend to a connection on $E$ with trivial holonomy. }\end{center}

In typical applications either the bundle $E$ is non-trivial, or else one of the  $\alpha_i$ has non-trivial holonomy,  and so the assumption holds. 
  With this assumption, the singularities of $\calm(E,\alpha)$ are cones on $\C P^n$, where $2n+1=\Ind^+(E,\alpha)$. If $n>0$ the second Stiefel-Whitney class of the base point fibration  $\mathcal{E}\to \calm(E,\alpha)^*$ restricts on the link  of each reducible  instanton to the generator of $H^2(\C P^n;\Z/2)=\Z/2$ \cite[Section 9]{FSpseudo}.

\medskip

 Suppose there exists a reducible connection $A\in \cala^\delta(E,\alpha)$.  Since $\Gamma_A$ centralizes the holonomy group of $A$, the assumptions that $(E,\alpha)$ is non-trivial and $H^1(X;\Z/2)=0$   imply that $\Gamma_A$ is    conjugate to $SO(2)\subset SO(3)$. 
Moreover, for each boundary component $Y_i$, $\alpha_i$ is a reducible flat connection on $Y_i$ whose stabilizer contains a maximal torus.

The action of $\Gamma_A$ on the fiber $E_x$ of $E$ over a point $x\in X_\infty$ gives a splitting $E_x=L_x\oplus \epsilon_x$, where $\epsilon_x$  is the 1-eigenspace of $\Gamma_A$ and $L_x$ is its 2-dimensional orthogonal complement on which 
 $\Gamma_A$ acts non-trivially.   
    Parallel transport using $A$ then gives a decomposition $E=L_A\oplus \epsilon_A$ into the sum of an orthogonal plane bundle bundle $L_A$ and a  real line bundle $\epsilon_A$. The bundle $\epsilon_A$ is trivial since $H^1(X;\Z/2)=0$.   A choice of trivialization of $\epsilon_A$ together with the orientation of $E$ determines an orientation  (and hence the structure of an $SO(2)$ vector bundle)    on $L_A$.  The connection $A$ splits correspondingly as the direct sum of connections $A=a\oplus \theta$, where $a$ is an $SO(2)$  connection on $L_A$ and $\theta$ is the trivial connection on the trivial bundle.   In particular, the connection  $\alpha_i$ on each end  splits as
 $$\alpha_i=\beta_i\oplus \theta \text{ on } L_A|_{Y_i}\oplus \epsilon,$$
 where $\beta_i=\lim_{r\to\infty} a|_{\{r\}\times Y_i}$.
 The  flat  connection $\beta_i$ is essentially uniquely determined by the subbundle $L_A|_Y\subset E|_Y$, as the following lemma shows.

 \begin{lemma}\label{unique} Every $SO(2)$ vector bundle $L$ over a rational homology sphere $Z$  admits a flat $SO(2)$ connection.  Moreover, 
given any two  flat $SO(2)$ connections $\beta,\ \beta'$ on   $L$, there exists a path of $SO(2)$ gauge transformations $g_t$ with $g_0$ the identity and $g_1^*(\beta')=\beta$.
 \end{lemma}

\begin{proof}  For simplicity identify $SO(2)$ with $U(1)$.  Choose any $U(1)$ connection $b$ on $L$.  Since $H^2(Z;\R)=0$, the real form $ \tfrac{i}{2\pi}F(b)$ is exact.     Adding  a    1-form $i\omega$ to $b$ changes $F(b)$ to $F(b)+id\omega$, and so there exists an $\omega$ for which $F(b+i\omega)=0$, i.e. $\beta=b+i\omega$ is flat.

The group of gauge transformations of the bunlde $L$, $\calg(L)=C^\infty(Z,U(1))$, has a single path component, since   $[Z,U(1)]=H^1(Z;\Z)=0$.  A flat connection $\beta$ on $L$ has a holonomy representation
$$\text{hol}_{\beta}\in \chi(\pi_1Z, U(1)) =
 \Hom(\pi_1Z, U(1))=H^1(Z;\R/\Z).$$ The Bockstein in the coefficient exact sequence $H^1(Z;\R/\Z)\to H^2(Z;\Z)$ takes $\text{hol}_{\beta}$ to $c_1(L)$  (see below) and is an isomorphism since $Z$ is a rational homology sphere.   Since the holonomy determines the gauge equivalence class of a flat connection, it follows that any two flat connections $\beta,\beta'$ on $L$ are gauge equivalent, and that there is a path  $g_t$ of gauge transformations with $g_0=Id$ and $g_1^*(\beta')=\beta$.
  \end{proof}

The relationship between the  holonomy $U(1)$ representation and the first Chern class used in the proof of Lemma \ref{unique} can be explained in the following way. Let $U(1)_d$ denote $U(1)$ with the discrete topology. The identity map $i:U(1)_d\to U(1)$   induces a map $Bi:BU(1)_d\to BU(1)$ on classifying spaces. Since  $BU(1)_d=K(\R/\Z ,1)$ and $BU(1)=K(\Z,2)$, we have a sequence of (natural) isomorphisms  for any finite abelian group $A$

\begin{equation*}\begin{split}
\Hom(A,U(1)) = H^1(K(A,1);\R/\Z)&=[K(A,1), BU(1)_d]\\
&\xrightarrow{Bi_*}[K(A,1),BU(1)]=H^2(K(A,1);\Z).
\end{split}\end{equation*}
The identity map $BU(1)\to BU(1)$, viewed as a class in $H^2(BU(1);\Z)$, is the first Chern class of the universal line bundle, and so naturality implies that a homomorphism $h:A\to U(1)$ is sent to  $c_1(L)$, where   $L$ is the pull back of the universal bundle via $ K(A,1)\xrightarrow{Bh} K(U(1)_d,1) \xrightarrow{Bi}BU(1)$.  Moreover, the map $Bi_*$ can be identified with the Bockstein.

Given a manifold $W$ equipped with a homomorphism $h:\pi_1(W) \to U(1)$ with finite image, view $h$ as an element of $H^1(W;\R/\Z)$. Then $h$ is sent by the Bockstein to the first Chern class of the $U(1)$ bundle that supports a flat connection with holonomy $h$.

In the case of a rational homology 3-sphere $Z$,  we have a sequence of isomorphisms
\begin{equation*}
\Hom(\pi_1(Z), \Q/\Z)=\Hom(H_1(Z),\Q/\Z)=H^1(Z;\Q/\Z)\xrightarrow{B}H^2(Z;\Z)\xrightarrow{\cap [Z]} H_1(Z;\Z)  
\end{equation*}
with $B$ the Bockstein and the last isomorphism Poincar\'e duality.  The inverse of this composite is given by the linking form,   $m\mapsto \lk(m,-):H_1(Z;\Z)\to \Q/\Z.$   Thus we have shown the following.
\begin{lemma}  \label{bock} Suppose $Z$ is a rational homology 3-sphere and  $c_1\in H^2(Z;\Z)$ a class with Poincar\'e dual $m\in H_1(Z;\Z)$.  
Then the  holonomy of the flat $U(1)$ bundle  with first Chern class  $c_1$ is given by 
$\text{\rm hol}(\gamma)=\lk(m,\gamma)$, where
$$\lk:H_1(Z;\Z)\times H_1(Z;\Z)\to \Q/\Z \xrightarrow{\exp(2\pi i -)}U(1)$$ is 
the linking form.\qed
\end{lemma}

Identifying $U(1)$ and $SO(2)$ via $\C=\R^2$ allows us to make the analogous statement, substituting the Euler class for $c_1$.

  \medskip

Suppose that $(E,\alpha)$ is an adapted bundle which admits a reducible connection $A$.  Then 
$(E,\alpha)=(L\oplus \epsilon, \beta\oplus \theta)$, where $\beta$ is unique in the sense of Lemma \ref{unique}.   Conversely, if $\epsilon\subset E$ is a trivializable real line bundle and $L$ the orthogonal 2-plane bundle, then  orient $\epsilon$ (and hence $L$). Let $\beta$ denote the flat $SO(2)$ connection carried by $L$  over the  end and  set $\alpha=\beta\oplus \theta$. Extend $\beta$ to a connection $a$ on $L$.  Then $A=a\oplus \theta$ is a reducible connection in $\cala(E,\alpha)$.  This establishes  the first assertion in the following lemma.

\begin{lemma}\label{support}
An adapted bundle $(E,\alpha)$ supports  a reducible connection if and only if there exists a nowhere zero section $s$ of $E$, and hence a splitting $E=L\oplus \epsilon$ with $L$  the orthogonal complement  of $\epsilon=\spa s$,
and, on the ends, $\alpha=\beta\oplus \theta$ where $\beta$ is the (unique up to homotopy) flat connection on $L$.

If an adapted bundle $(E,\alpha)$ supports a reducible connection inducing a splitting $E=L\oplus \epsilon$, it supports a reducible instanton in $\calm(E,\alpha)$ inducing the same splitting.
 \end{lemma}

 \begin{proof}     The first part was justified in the preceding paragraph.
Suppose that $(E,\alpha)=(L,\beta\oplus \theta)$. 
Since $Y$ is a union of rational homology spheres,  Proposition 4.9 of \cite{APS1} shows that the  space of $L^2$ harmonic 2-forms   $ \{\omega\in L^2(\Omega^2_{X_\infty})  |   d\omega=0=d^*\omega\}
$  maps isomorphically to $H^2(X;\R)$.    Thus there exists an $L^2$ harmonic  2-form $\omega$ on $X_\infty$ representing the class $e(L)\in H^2(X;\R)$.  In light of the facts that $X$ is negative definite, the Hodge $*$-operator preserves  $L^2$ harmonic 2-forms,  and $*^2=1$, it follows that $*\omega=-\omega$.   Choose a connection $a$ on $L$. Then $\tfrac{1}{2\pi}\Pf(F(a))$ and $\omega$ are cohomologous by Chern-Weil Theory, where Pf denotes the pfaffian of an $so(2)$ matrix. By adding an $so(2)$ valued 1-form to $a$ one can change $\Pf(F(a))$ arbitrarily within its cohomology class.     Thus  we may choose $a$ so that $\tfrac{1}{2\pi}\Pf(F(a))=\omega$.  

The connection $a$ limits to a flat connection on $L$ along the ends.   Thus Lemma \ref{unique} implies $a$ limits to $h^{-1}(\beta)$ for some gauge transformation $h:L|_Y\to L|_Y$ which  can be connected to the identity by  a path of gauge transformations.   Extend $h$ to a gauge transformation $\tilde h$ of all of $L$ by the identity on $X$, the path from the identity to $h$ on $[0,1]\times Y$,  and $h$ on $[1,\infty)\times Y$.

Replacing $a$   by $\tilde h(a)$ does not change $F(a)$ since $SO(2)$ is abelian, and $h(a)$ limits to $\beta$.  Let $A=\tilde h(a)\oplus \theta$ on the $SO(3)$ vector bundle $L\oplus \epsilon$. 
Then $A$ is a reducible instanton on $(E,\alpha)=(L\oplus \epsilon,\beta\oplus\theta)$ with $L^2$-curvature.  
Since we assumed that the flat connection $\alpha$  on $Y$ is non-degenerate,   it follows  from \cite[Theorem 4.2]{donald1} that   $A\in L^{p,\delta}_1$ and so $A\in \calm(E,\alpha)$.
 \end{proof}

  The following proposition determines the Pontryagin charge of a reducible instanton in terms of the intersection form.

\begin{prop} \label{lem2.4} Let   $e\in H^2(X;\Z)$. Let $L\to X_\infty$ be the $SO(2)$ vector bundle with Euler class $e$.   Let  $\beta$ be the  flat connection on the restriction $L|_Y$, and $(E,\alpha)=(L\oplus \epsilon, \beta\oplus \theta)$ the corresponding adapted bundle. 

 Then
$$p_1(E,\alpha)=-e\cdot e\in \tfrac{1}{d}\Z\text{ and } w_2(E)\equiv e \text{\rm  \ mod } 2.$$ 
 \end{prop}  

 \begin{proof}   Choose (arbitrarily) a connection $a_0$ on $L$ which extends the flat connection $\beta$ on the ends, and let $A_0=a_0\oplus \theta$ be the corresponding connection on $E=L\oplus \epsilon$; this connection agrees with   $\alpha=\beta\oplus \theta$ over the ends, and hence lies in $\cala^\delta(E,\alpha)$.

Let $\hat e$ denote the differential 2-form on $X_\infty$  
$$\hat e= \tfrac{1}{2\pi}{\text{\rm Pf}}(F(a_0)).$$
 A straightforward calculation starting with the formula for the inclusion of Lie algebras $  so(2)\subset so(3)$ gives
\begin{equation*}\begin{split}
-\frac{1}{8\pi^2}\Tr(F(A_0)\wedge F(A_0))& =-\tfrac{1}{2}\Tr(\tfrac{1}{2\pi}F(A_0) \wedge \tfrac{1}{2\pi}F(A_0))\\
&=-\tfrac{1}{2\pi}{\text{\rm Pf}}(F(a_0) )\wedge \tfrac{1}{2\pi}{\text{\rm Pf}}(F(a_0))\\
&=-\hat{e}\wedge \hat{e}.
\end{split}\end{equation*}
Hence $$p_1(A_0)=p_1(E,\alpha)= -\int_{X_\infty} \hat{e}\wedge \hat{e}.$$

 Chern-Weil theory implies that $\hat e$ is a closed 2-form. It vanishes on the ends  since $a_0$ is flat on the ends.  On a {\em closed}   4-manifold  $Z$  Chern-Weil theory implies that $\hat e$ represents (in DeRham cohomology) the image of the Euler class $e$ under the coefficient homomorphism $H^2(Z;\Z)\to H^2(Z;\R)$.  Since $H^2(X;\R)=\Hom(H_2(X;\Z),\R)$     and every class in $H_2(X;\Z) $ is represented by a closed surface, it follows that $\hat e$ represents the image of $e$ in $H^2(X;\R)$.

For  $x,y\in H^2(X,Y;\R)$ represented by closed forms that vanish on the boundary, $x\cdot y= \int_{X}x\wedge y$.  Since $H^2(X,Y;\R)\to H^2(X;\R)$ is an isomorphism, it follows that 
 $$p_1(E,\alpha)=-\int_{X} \hat{e}\wedge \hat{e}=-e\cdot e.$$
 The congruence $w_2(E)\equiv e \text{\rm  \ mod } 2 $ follows from the general congruence $w_2(L\oplus \epsilon)\equiv e(L)$ mod $2$, which is valid for any $SO(2)$ bundle $L$ over any space.  
 \end{proof}

\subsection{Enumeration of reducible instantons}
  We next address the problem of enumerating the reducible instantons in $\calm(E,\alpha)$. The results of the previous section allow us to  reduce this problem to the enumeration of bundle reductions.  These in turn determine certain cohomology classes and motivate the following definition.

\begin{definition} Let $(E,\alpha)=(L\oplus\epsilon,\beta\oplus \theta)$ with $L$ non-trivial. Fix an orientation of $\epsilon$, orienting $L$. Let $e=e(L)\in H^2(X;\Z)$ denote the Euler class of $L$. 
 Define
  $$C(e)=\{e'\in H^2(X;\Z)\ | \ e'\cdot e'=e \cdot e , \ e'\equiv e \text{\rm \ mod } 2,\     i_j^*(e')=\pm i_j^*(e),\ j=1,\cdots,c  \}/\scriptstyle{\pm 1}
  $$
    where  $i_j:Y_j\subset Y\subset X$ denotes the inclusion of the $j$th boundary component. Our standing assumption that $X$ is negative definite implies that $C(e)$ is a finite set. 
  \end{definition}

\medskip

   Suppose that $L,L'$ are two  orthogonal plane  subbundles of $E$ and that $(E,\alpha)=(L\oplus \epsilon,\beta\oplus \theta)=(L'\oplus \epsilon',\beta'\oplus \theta)$.  Orient $L$ and $L'$ arbitrarily and let $e=e(L)$, $e'=e(L')$.
 Then $e\equiv e'$ mod 2 and $e\cdot e=e'\cdot e'$ by Proposition \ref{lem2.4}.   
Fix a boundary component $Y_j$ of $X$. If $\alpha_j$ is non-trivial, then since $\beta_j\oplus\theta=\alpha_j=\beta_j'\oplus \theta$, the {\em unoriented subbundles} 
$L|_{Y_j}$ and $ L'|_{Y_j}$   are equal, since their orthogonal complements coincide.  But if $\alpha_j$ is the trivial connection, $L|_{Y_i}$ and $ L'|_{Y_i}$  need not coincide, although they are both trivial bundles. 
In either case, $i_j^*(e)=\pm   i_j^*(e')$.  Changing the orientation of $L'$ changes the sign of $e'$.     Since $X$ is connected,   it follows that the  function taking a reducible connection $A\in\cala^\delta(L\oplus\epsilon, \beta\oplus \theta)_{red}$ to $e(L_A)\in C(e) $ is well-defined.

  \begin{prop} \label{reducibles}   The  function $\Phi:\calm(E,\alpha)_{red}\to C(e)$ defined by $\Phi([A])=e(L_A)/\scriptstyle{\pm1}$ is well-defined and injective.
If $C(e)$ consists of a single point then $\Phi$ is a bijection.
 \end{prop}
  \begin{proof}

We first show that $\Phi$ is well-defined  on gauge equivalence classes.   Suppose that $A$ is a reducible  instanton on $(E,\alpha)$ and $g\in \calg^\delta(E,\alpha)$ a gauge transformation.  Then the decomposition 
$E=L_A\oplus \epsilon$ is sent to $E=g(L_A)\oplus g(\epsilon)$.  Hence $L_A$ and $g(L_A)$ are isomorphic bundles over $X_\infty$, and so $\Phi(A)=\Phi(g(A))$.  Thus $\Phi$ is well-defined on gauge equivalence classes.

Suppose that $A$ and $B$ are reducible instantons on $(E,\alpha)$ which satisfy $\Phi(A)=\Phi(B)$. Then, after reorienting $L_B$ and $\epsilon_B$ if needed, $e(L_A)=e(L_B)$.  Since $SO(2)$ bundles are determined up to isomorphism by their Euler class, there exists a bundle isomorphism $h:L_A\to L_B$.  Extending by the unique orthogonal orientation-preserving isomorphism $I:\epsilon_A\to\epsilon_B$  produces a bundle isomorphism
$g=h\oplus I:E=L_A\oplus \epsilon_A\to L_B\oplus \epsilon_B=E$.  Since the decompositions agree on the ends with the given decomposition $(E,\alpha)$, $\lim_{r\to \infty} g|_{\{r\}\times Y_j}$ exists and lies in the stabilizer of $\alpha_j$, so that $g\in \calg^{\delta}(E,\alpha)$.  Therefore, $A$ and $B$ represent the same point in $\calm(E,\alpha)$.

  Lemma \ref{support} shows that $\calm(L\oplus \epsilon,\beta\oplus \theta)$ is non-empty, and so 
if $C(e)=\{e\}$,  $\Phi$ is surjective.  
   \end{proof}

Determining the image of $\Phi$ can be tricky  when $Y$ is not a union of $\Z/2$ homology spheres.
Lemma \ref{support} shows that $e=e(L)$ itself is always in the image, and hence, as noted in Proposition \ref{reducibles},  if $C(e)$ consists of a single point, then $\Phi$ is a bijection.

\begin{theorem}  \label{CofL}  Let $(E,\alpha)=(L\oplus \epsilon, \beta\oplus \theta)$ be a reducible adapted bundle with Euler class $e=e(L)$ and let $\Phi:   \calm(E,\alpha)_{red}\to C(e) $ be the injective function defined above.
\begin{enumerate}
\item If    $H^2(X;\Z)$ splits orthogonally with respect to the intersection form as $H^2(X;\Z)=\spa(e)\oplus W$ with $\spa(e)\cong\Z$ and $W$ free abelian, then $C(e)$ consists of a single point and so $\Phi$ is a bijection. 
\item To each $e'\in C(e)$ one can assign  an obstruction in $H^1(Y;\Z/2)$ whose vanishing implies that $e'=\Phi([A]) $ has a solution $[A]\in\calm(E,\alpha)_{red}$. 
\item If each component of $Y$ is a $\Z/2$ homology sphere, then $\Phi$ is a bijection.
\end{enumerate}
\end{theorem}

\begin{proof}  

Suppose first that $H^2(X;\Z)$ splits orthogonally with respect to the intersection form as $H^2(X;\Z)=\spa(e)\oplus W$.  Then any class $e'$ satisfying $e'\equiv e$ mod 2 takes the form $e'= (1+2k)e + 2w$, with $w\in W$.  Hence $e'\cdot e'=(1 + 4k+4k^2)e\cdot e +4w\cdot w$.  If $e'\in C(e)$ then $e'\cdot e'=e \cdot e$, so that $4(k^2+k) e\cdot e + 4w\cdot w=0$.  Since $e\cdot e<0$ and $w\cdot w\leq 0$,  this is possible if and only if $k=0$ or $-1$ and $w=0$, so that $e'=\pm e$, as claimed.   

We turn now to the second assertion, which clearly implies the third assertion.

Given a class $e'\in C(e)$, let $L'\to X$ be an $SO(2)$ vector bundle with $e'=e(L')$.  Then since $e'\equiv e$ mod 2, the bundles $E=L\oplus \epsilon$ and $L'\oplus \epsilon'$ over $X$ are isomorphic.     Fix an identification $E=L\oplus \epsilon= L'\oplus \epsilon'$. 
Let $\beta_j$ be the unique flat connection on $L|_{Y_j}$,  $\beta'_j$ be the unique flat connection on $L'|_{Y_j}$. Then $\alpha_j=\beta_j \oplus \theta$ and $\alpha_j'=\beta_j'\oplus \theta$ are two flat connections on $E|_{Y_j}$.

If there exists a gauge transformation $g:E\to E$ so that $g|_{Y_j}^*(\alpha'_j)= \alpha_j$, then Lemma \ref{support} implies that there exists an instanton   $A'\in \calm(E,\alpha)$ inducing the splitting   $E=g(L')\oplus g(\epsilon')$, and hence $e'=\Phi({A'}) $ has the solution $A'$.  So we seek to construct such a gauge transformation $g$. We construct $g$ first on the boundary and the try to extend over the interior.

For each $j$, choose an isomorphism $h_j:L|_{Y_j}\to L'|_{Y_j}$ which preserves or reverses orientation according to whether $i_j^*(e')=i_j^*(e)$ or 
$i_j^*(e')=-i_j^*(e)$ (if $ 2i_j^*(e)=0$ choose $h_j$ to preserve orientation, for definiteness).  Then $h^*_j(\beta_j')$ and $\beta_j$ are flat connections on $L|_{Y_j}$ and so we may assume, by modifying $\beta_j'$ if needed, that  
$h^*_j(\beta_j')=\beta_j$. Let $g_j:E|_{Y_j}=L|_{Y_j}\oplus \epsilon\to E= L'|_{Y_j}\oplus \epsilon'$ be the unique orientation-preserving orthogonal extension   (so that on the trivial $\R$ factors, the orientation is preserved or reversed according to the sign).  By construction, $g_j^*(\alpha_j')=\alpha_j$.  
Let $g  =\sqcup_j g_j:E|_Y\to E|_Y$.

  The stabilizer of $\alpha_j$ is connected except for those $j$ for which $i_j^*(e)$ has order 2, in which case the stabilizer is isomorphic to $O(2)$. Thus the desired gauge transformation exists if and only if $g$ can be extended 
 to a gauge transformation over $X$, after perhaps composing $g_j$ with an appropriate 
gauge transformation $k_j$ for those path components such that $i_j^*(e)$ has order 2.

Let $P$ be the principal $SO(3)$ bundle associated to $E\to X$, so that gauge transformations are sections of $ P\times_{Ad}SO(3)$.    
Obstruction theory identifies the  primary obstruction   to extending $g$ to all of $ X$ as a class  $\psi \in H^2(X,Y;\pi_{1}(SO(3)))=H^2(X,Y;\Z/2)$. Naturality of primary obstructions, together with the fact that $ P\times_{Ad}SO(3)$ admits sections (e.g. the identity transformation) shows that $\psi$ is in $  \ker H^2(X,Y;\Z/2)\to H^2(X;\Z/2)$, hence $\psi$ lies in the  image of the injection $H^1(Y;\Z/2)\to H^2(X,Y;\Z/2)$.  Let $\vartheta\in H^2(Y;\Z/2)$ denote the unique element mapped to $\psi$.

When $\vartheta=0$, then $g$ extends over the 2-skeleton of $X$.  It further extends further to the  3-skeleton,  because
$\pi_2(SO(3))=0$.  The last obstruction to extending $g$ over $X$ lies in $H^4(X,Y;\pi_3(SO(3)))=H^4(X,Y;\Z)$; the coefficients are untwisted since our standing assumption that $H^1(X;\Z/2)=0$ implies that $\pi_1(X)$ admits no non-trivial homomorphisms to Aut$(\Z)=\Z/2$.   This class is the obstruction to extending $g$  over the top cell of $X/Y$.
   
Suppose that $D \subset X$ is a fixed 4-ball, and that $g:E|_{X\setminus D}\to E|_{X\setminus D}$ is  the gauge transformation which we wish to extend to all of $E$.   Choose a connection  $A\in \cala^\delta(L\oplus \epsilon,\beta_L\oplus \theta)$.  The connection $g^*(A|_{X\setminus D})$  extends to a connection $ B\in \cala^\delta(L'\oplus \epsilon', \beta_{L'}\oplus \theta')$.    

Since $\Tr(F(g^*(A))\wedge F(g^*(A)))=\Tr(F(A)\wedge F(A))$ pointwise (on $X\setminus D$), it follows that 
  \begin{equation}
\label{eq2.3}
p_1(A)-p_1(B)=\tfrac{1}{8\pi^2}\int_{D}\Tr(F(A)\wedge F(A))-\tfrac{1}{8\pi^2}\int_{D}\Tr(F(B)\wedge F(B)),
\end{equation}
The condition $e'\cdot e'=e\cdot e$ implies that 
 $p_1(A)=p_1(B)$, and hence the left side of Equation (\ref{eq2.3})  is zero. 
 
 But  Chern-Weil theory implies that  the right side of Equation (\ref{eq2.3}) equals $p_1(V)$, where $V$ is the 
 bundle over $S^4$ obtained by clutching two copies of $E|_D$ using $g|_{\partial D}$.  Since $p_1(V)=0$,  the bundle $V$ is trivial, and   thus  the restriction of $g$ to $\partial  D$  extends over $D$, giving the desired gauge transformation $g$. 
 
 Thus we have seen that there is a class $\vartheta\in H^1(Y;\Z/2)$ whose vanishing guarantees that $g$ extends over $X$.
 \end{proof}

 \subsection{The argument of Furuta}
 
 With all the necessary machinery in place, we can now describe our variant of  the argument of Furuta. 
 Theorem \ref{furfinster} provides the mechanism to establish that certain  definite 4-manifolds cannot exist. Furuta used this argument to prove the linear independence of the Seifert fibered homology spheres $\{\Sigma(p,q,pqk-1)\}_{k=1}^\infty$ in $\Theta^3_H$.  We will give further applications below.

Recall from Section \ref{reduciblesect} that if $X$ is a negative definite four-manifold satisfying $H^1(X;\Z/2)=0$, and  $\partial X$ consists of a union of rational homology spheres, then  any   $SO(2)$ bundle $L$ determines a unique adapted $SO(3)$ 
 bundle $(E,\alpha)=(L\oplus \epsilon,\beta\oplus \theta)$. 
Letting $e(L)$ denote the Euler class of $L$,   Proposition \ref{lem2.4} shows that   $p_1(E,\alpha)=-e(L)\cdot e(L)\in \tfrac{1}{d}\Z$, where $d$ is an integer which annihilates $H_1(Y_i;\Z)$ for each $i$.  Finally, recall the invariant  $\hat\tau(Y,\alpha)\in(0,4]$ of Definition  \ref{defoftau}.

 \begin{theorem}\label{furfinster}  Let $X$ be a negative definite 4-manifold with $H^1(X;\Z/2)=0$ whose boundary consists of a union of rational homology spheres.   Let  $L$ be an $SO(2)$ bundle over $X_\infty$ with Euler class $e=e(L)$,  and let $(E,\alpha)=(L\oplus \epsilon,\beta\oplus \theta)$ be the corresponding adapted bundle.  Assume that 
 \begin{enumerate}
 \item Each $\alpha_i=\beta_i\oplus \theta$ is non-degenerate,
\item $\Ind^+(E,\alpha)\ge 0$, where 
$$ \Ind^+(E,\alpha)=-2e\cdot e-3 +\tfrac{1}{2}\sum_{\beta_i \text{\tiny \rm nontrivial}}(3-h_{\alpha_i}-\rho(Y_i,\alpha_i)).$$
\item  $0<-e\cdot e<\hat\tau(Y,\alpha)\leq4$.
\end{enumerate}

 Then the number of gauge equivalence classes of reducible instantons on $(E,\alpha)$  is even, greater than 1, and $\Ind^+(E,\alpha)$ is odd.  \end{theorem}
 \begin{proof}  The irreducible moduli space
 $\calm(E,\alpha)^*$ is a smooth oriented manifold of dimension $\Ind^+(E,\alpha)$.    Let $k$ denote the number of singular points in the full moduli space.  Proposition   \ref{reducibles} identifies  $k$ with the image of $\Phi:\calm(E,\alpha)_{red}\to C(L)$, and shows that $k\ge 1$. Each singular point $A\in \calm(E,\alpha)_{red}$ has a neighborhood homeomorphic to a cone on $\C P^n$ for some $n$, and so $2n+1=\Ind^+(E,\alpha)$.   Proposition  \ref{compactness} implies that $ \calm(E,\alpha)$ is compact.   
 
Removing cone neighborhoods of the $k$ singular points yields a $2n+1$ dimensional compact manifold $\calm^0$,  whose boundary consists of a disjoint union of $k$ copies of $\C P^n$ (perhaps with differing orientations).  This manifold is equipped with an $SO(3)$ bundle $\mathcal{E}\to \calm^0$ that restricts to a bundle on each $\C P^n$ with non-trivial second Stiefel-Whitney class. Since $H^*(\C P^n;\Z/2)$ is a polynomial ring on $w_2$, $0\ne w_2(\mathcal{E}|_{\C P^n})^{n}$ for each boundary component. But since $(\sqcup_{i=1}^k \C P^n, \mathcal{E})$ extends to $(\calm^0, \mathcal{E})$, it follows  that $k$ is even, as this is the only way for the Stiefel-Whitney numbers of  $(\sqcup_{i=1}^k \C P^n, \mathcal{E})$ to vanish.
\end{proof}

 \subsection{Computing Atiyah-Patodi-Singer  $\rho$ invariants and Chern-Simons invariants of flat  SO(2)  connections on rational homology spheres}\label{rhoinvts}
 
Before we turn to applications and examples we discuss the use of flat cobordisms to compute the  $\rho$ invariants and Chern-Simons invariants which appear in the expression  for $\Ind^+(E,\alpha)$  and the definition of $\hat\tau(Y,\alpha)$.
 
 \medskip
 
The Atiyah-Patodi-Singer $\rho$ invariants of flat $SO(3)$ connections on 3-manifolds can be difficult to compute. However, in Theorem \ref{furfinster} we require $Y$ to be a union of rational homology spheres and the flat connection $\alpha$ to be reducible;  in fact,  to take the form $\alpha=\beta\oplus \theta$ for $\theta$ the trivial connection. Thus, what is required in applications of this theorem is a calculation of  $\rho$ invariants of  such reducible  flat $SO(3)$ connections on rational homology spheres.  This is an easier task, and can often be accomplished as follows.  

Since $\rho(Y,\alpha)$ depends only on the gauge equivalence class of $\alpha$, and taking the holonomy induces a 1-1 correspondence from  gauge equivalence classes of flat connections to conjugacy classes of  homomorphisms $\pi_1(Y)\to SO(3)$, we henceforth use the notation  $\alpha$ for both a flat connection and its holonomy.  

Then $\alpha:\pi_1(Y)\to SO(3)$ is reducible if  it is   conjugate to a   representation $\beta\oplus 1:\pi_1(Y)\to SO(3)$ with image in 
$$\begin{pmatrix}
SO(2)&0\\0&1 
\end{pmatrix} \subset SO(3).$$
Complexification preserves eigenvalues, and hence $\rho(Y,\alpha)$ is  unchanged by   including $SO(3)$ in $SU(3)$.  The  complexification of $\beta\oplus 1$ is conjugate to  the diagonal representation $\beta_\C\oplus \bar\beta_\C\oplus 1$ where $\beta_\C:\pi_1(Y)\to U(1)$ is just $\beta$ viewed via the identification $SO(2)=U(1)$.   Because $\rho(Y,\theta)=\eta(Y,\theta)-\eta(Y,\theta)=0$,   
 \begin{equation}\label{eq2.9}
\rho(Y,\alpha)=\rho(Y,\beta_\C)+   \rho(Y,\bar\beta_\C).
\end{equation}
 Therefore the task is reduced to computing $\rho(Y,\gamma)$ for 1-dimensional unitary  representations $\gamma$.

A successful strategy to compute $\rho(Y,\gamma)$ for $\gamma:\pi_1(Y)\to U(1)$ is to find
a cobordism $W$ from $Y$ to a lens space (or union of lens spaces) $L(Y)$ over which $\gamma$  extends. Finding such a cobordism is a familiar problem; see for example \cite{casson-gordon,KKR}.  The key idea is that  finding such a $W$ is equivalent to showing $(Y,\gamma)$ vanishes in an appropriate oriented bordism group over $\R/\Z$ or $\Z$.  This can be addressed with straightforward   methods.

Such a  $W$ and extension of $\gamma$ provides  a {\em flat} cobordism from $Y$ to $L(Y)$. Applying the Atiyah-Patodi-Singer  signature theorem \cite{APS1, APS2} one obtains:
$$ \rho(Y, \gamma)-\rho(L(Y),\gamma)=\text{Sign}_{\gamma}(W)- \text{Sign}  (W).$$
By this approach, computing the $\rho$ invariants  which arise in Theorem \ref{furfinster}    reduces to computing $U(1)$ signatures  of $W$ and computing  $U(1)$ $\rho$ invariants of lens spaces.

\medskip 

The calculation  of  $\rho(L(a,b),\gamma)$ for a lens space $L(a,b)$ and a $U(1)$ representation $\gamma:\pi_1(L(a,b))\to U(1)$ 
  has been carried out in   \cite{donelly, APS2}.    We state the result in the form we need.   Let $a$ be a positive integer and $0<b<a$ an integer relatively prime to $a$.  Denote by $L(a,b)$ the oriented lens space obtained by $-a/b$ surgery on the unknot $U$.  
The meridian $\mu$ of the unknot generates $\pi_1(S^3\setminus U) $, and hence also $\pi_1(L(a,b))=H_1(L(a,b))=\Z/a$.    

\begin{prop} \label{prop2.13}  Let $\beta:\pi_1(L(a,b))\to SO(2)$ be the representation sending the generator $\mu$ to the rotation of angle $\frac{2\pi  \ell}{a}$  for some integer $\ell$.
Then 
$$\rho(L(a,b),\beta)  =\tfrac{4}{a}\sum_{k=1}^{a-1}\cot(\tfrac{ \pi k}{a})\cot(\tfrac{ \pi k b}{a})  \sin^2(\tfrac {\pi k\ell}{a}). $$
If $\alpha=\beta\oplus \theta:\pi_1(L(a,b))\to SO(3)$, then $\rho(L(a,b),\alpha)=\rho(L(a,b),\beta)$.
  \end{prop} 
 \begin{proof}
Equation  (\ref{eq2.9}) shows that  
$$\rho(L(a,b),\beta)=\rho(L(a,b),\beta_\C)+\rho(L(a,b),\bar\beta_\C)$$
where $\beta_\C:\pi_1(L(a,b))\to U(1)$ is given by $\beta_\C(\mu)=e^{2\pi i\ell/a}$.

The lens space $L(a,b)$ is orientation-preserving diffeomorphic to the quotient of $S^3$ by the free action $g\cdot(z,w)=(e^{2\pi i/a}z, e^{2\pi i b/a}w)$ for $(z,w)\in S^3\subset \C^2$.  This diffeomorphism identifies the homotopy class of $\mu$ with the generator $g$ 
of the deck transformations.     Proposition 4.1 of \cite{donelly} calculates the corresponding $\eta$ invariant; the formula is
$$\eta(L(a,b),\beta_\C)=-\tfrac{1}{a}\sum_{k=1}^{a-1}\cot(\tfrac{ \pi k}{a})\cot(\tfrac{ \pi k b}{a})e^{2\pi i k\ell/a}.$$
Since $\rho(L(a,b),\beta_\C)+\rho(L(a,b),\bar\beta_\C)=\eta(L(a,b),\beta_\C)+\eta(L(a,b),\bar\beta_\C)-2\eta(L(a,b),\theta)$,
 \begin{eqnarray*}
\rho(L(a,b),\beta)&=&-\tfrac{1}{a}\sum_{k=1}^{a-1}\cot(\tfrac{ \pi k}{a})\cot(\tfrac{ \pi k b}{a})(e^{2\pi i k\ell/a}+e^{-2\pi i k\ell/a}-2)\\
&=&-\tfrac{2}{a}\sum_{k=1}^{a-1}\cot(\tfrac{ \pi k}{a})\cot(\tfrac{ \pi k b}{a})(\cos(\tfrac {2\pi k\ell}{a})-1)\\
&=& \tfrac{4}{a}\sum_{k=1}^{a-1}\cot(\tfrac{ \pi k}{a})\cot(\tfrac{ \pi k b}{a})\sin^2(\tfrac {\pi k\ell}{a}). 
\end{eqnarray*}
The last assertion follows from the facts that $\rho$ invariants add with respect to the direct sum of representations and that $\rho(Z,\theta)=\eta(Z,\theta)-\eta(Z,\theta)=0$ for any $Z$.
\end{proof}

The formula of  Neumann-Zagier  \cite{NZ} 
\begin{equation}\label{NZformula}\begin{split}\tfrac{2}{a}\sum_{k=1}^{a-1}\cot(\tfrac{\pi  kc}{a})\cot(\tfrac{ \pi k  }{a})  \sin^2(\tfrac {\pi k  }{a})& = \tfrac{2c^*}{a}-1 \\   \text{ when } (a,c)=1   \text{ and } 0<c^*<a &\text{ satisfies }   cc^*\equiv -1 \text{ mod } a,
\end{split}\end{equation}  
gives an explicit method for computing the $\rho$ invariants which appear in Proposition \ref{prop2.13}
 in the case when $\ell=b$ (by reindexing the sum).    In the general case the extensions of this calculation due to Lawson 
 \cite{lawson,lawson2} are useful.

 \medskip

The strategy of using a flat cobordism also plays an important role in computing Chern-Simons invariants and hence the invariant $\hat\tau(Y,\alpha)$.   This is because if $W$ is a flat cobordism from $(Y,\alpha)$ to $(Y',\alpha')$, then $\cs(Y,\alpha)=\cs(Y',\alpha')$ mod $\Z$.

 For {\em reducible} $\alpha$  the  flat cobordisms as described above permits one to express  Chern-Simons invariants in terms of Chern-Simons invariants of lens spaces.  Indeed, 
 one can quickly obtain estimates for Chern-Simons invariants of lens spaces (or, more generally, spherical space forms) as follows. First note that $L(a,b)$ has universal cover $S^3$, and every flat $SO(3)$ connection on $S^3$ is gauge equivalent to the trivial connection, which has Chern-Simons invariant a multiple of 4. Integrating the Chern-Simons form associated to the pullback of any flat connection over a fundamental domain shows that $${\tau}( L(a,b),\alpha )=\frac{4k}{a}\ge \frac{4}{a}.$$  
 
 Computing or estimating Chern-Simons invariants for irreducible flat $SO(3)$ connections, however, requires  more sophisticated techniques, such as those introduced in \cite{KK1, KK2,KKR, auckly}.

\section{Seifert fibered examples}\label{FSexample}

Propositions  \ref{APSindex}  and    \ref{prop2.13} allow us to identify 
 $\Ind^+(E,\alpha)$ with the invariant  $R(a_1,\cdots, a_n)$ of Fintushel-Stern \cite{FS}  for $X$ the truncated mapping cylinder of the  Seifert fibration of the Seifert fibered homology sphere $\Sigma(a_1,\cdots,a_n)$. We explain this and give their proof that when $R(a_1,\cdots, a_n)>0$,  $\Sigma(a_1,\cdots,a_n)$ does not bound a positive definite 4-manifold.

Fix an $n$-tuple $(a_1,b_1),\cdots, (a_n,b_n)$ of  relatively prime pairs of integers  satisfying $(a_1\cdots a_n)\sum_{i=1}^n\tfrac{b_i}{a_i}=1$ with the $a_i$ positive.  These uniquely define an oriented 
Seifert fibered homology sphere $\Sigma(a_1,\cdots,a_n)$ whose Kirby diagram is illustrated in Figure \ref{fig1}.
  
  Starting with the mapping cylinder $C$ of the Seifert fibration $\Sigma\to S^2$, remove  neighborhoods of the $n$ singular points of $C$ leaving a 4-manifold $X$. These neighborhoods are homeomorphic to cones on lens spaces $L(a_i,b_i)$.    Thus $X$ is a manifold with boundary, and this boundary consists of the disjoint union of $\Sigma$ and the lens spaces $L(a_i,b_i)$.    For convenience  we denote $Q=\sqcup_i L(a_i,b_i)$, so that $\partial  X=Q\sqcup \Sigma $.

 The 4-dimensional Kirby diagram for $X$ is also given by Figure \ref{fig1}.  More precisely, this diagram can  be thought of as a cobordism  from a connected sum of lens spaces $-L(a_i,b_i)$ to $\Sigma$, obtained by adding a single $0$-framed 2-handle.  Then $X$ is obtained by gluing this cobordism   to a  cobordism 
 from the disjoint union of lens spaces to the connected sum of lens spaces obtained by attaching   $(n-1)$ $1$-handles to form the connected sum.  
 Thus $X$ is  obtained from $I\times (-Q)$ by adding $(n-1)$ 1-handles and one 2-handle; alternatively $X$ is obtained from $I\times \Sigma$ by adding one 2-handle (a 0-framed meridian to the 0-framed circle indicated in Figure \ref{fig1}) and $(n-1)$ 3-handles along 2-spheres separating the components. Therefore $H_1(X)=0$ and $H_2(X)\cong \Z$.

\begin{figure}
\psfrag{a1}{$ \frac{a_1}{b_1}$}
\psfrag{a2}{$ \frac{a_2}{b_2}$}
\psfrag{a3}{$ \frac{a_3}{b_3}$}
\psfrag{a4}{$ \frac{a_n}{b_n}$}
\psfrag{0}{$0$}
\begin{center}
 \includegraphics [height=90pt]{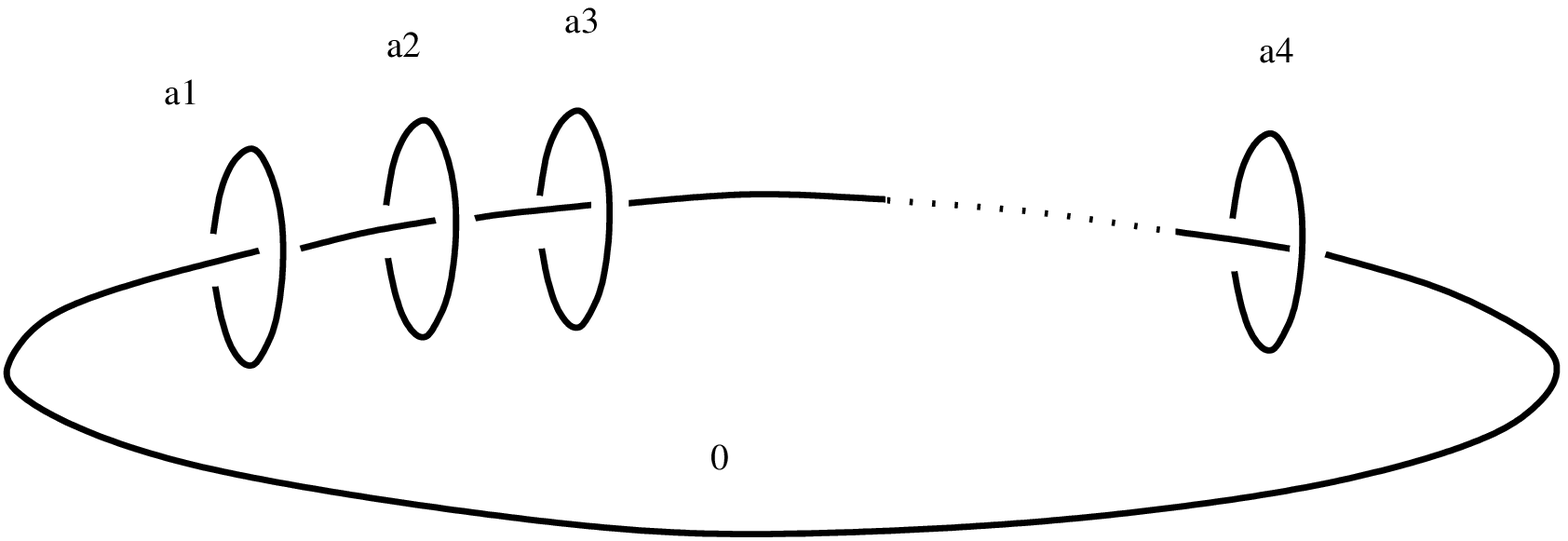}
\caption{}\label{fig1} 
\end{center}
\end{figure}

Consider the surface  obtained  from the union of the punctured disk bounded by the $0$-framed component together with the core of the 2-handle. Pushing the interior of this surface slightly into the interior of $X$ gives a properly embedded $n$-punctured sphere $F\subset X$ representing a generator $[F,\partial F]\in H_2(X,Q)=H_2(X,\partial X)\cong\Z$.  Orient $X$ so that it is negative definite. With this orientation, $\partial X=\Sigma\sqcup_i  L(a_i,b_i)$ (recall that as an oriented 3-manifold $L(p,q)$ is defined to be $-\frac{p}{q}$ Dehn surgery on the unknot).

 Let $L\to X$ be the $SO(2)$ bundle whose Euler class $e(L)\subset H^2(X)$ equals the 
Poincar\'e dual to $F$, and $(E,\alpha)=(L\oplus \epsilon,\beta\oplus \theta)$ the corresponding adapted bundle.   We compute the terms that appear in the formula for $\Ind^+(E,\alpha)$  given by Proposition \ref{APSindex}. 

  The product $a=a_1\cdots a_n$ annihilates $H_1(\partial X)$.  
The boundary components of $F$ are meridian curves representing the generators $\mu_i\in H_1(L(a_i,b_i))$. 
A straightforward geometric argument using Figure \ref{fig1} shows that 
$$F\cdot aF= - (a_1\cdots a_n)\sum_{i=1}^n\tfrac{b_i}{a_i}=-1$$
 Thus
$$p_1(E,\alpha)=-e(L)\cdot e(L)=- \tfrac{1}{a}(F\cdot aF)= \tfrac{1}{a}.$$

 Next, note that the flat $SO(2)$ connection on $\Sigma$ is trivial since $\Sigma$ is a integer homology sphere, or, equivalently, $F$ misses $\Sigma$.

 This leaves the contributions from the $L(a_i,b_i)$.   We use Proposition \ref{prop2.13}  to compute  the terms $ \rho(L(a_i,b_i),\alpha_i)$.   The representation $\alpha_i=\beta_i\oplus \theta$ sends $\mu_i$ to $e^{2\pi i \ell_i/a_i}$ where, according to Lemma \ref{bock},  the  integers $\ell_i$ are determined  by $\ell_i/a_i=\lk(\mu_i, \mu_i)$. It follows that $\ell_i=-b_i$ and so $\beta_i$ is non-trivial.  Hence $H^0(L(a_i, b_i);\R^2_{\beta_i})=0=H^1(L(a_i, b_i);\R^2_{\beta_i})$ and so
 $$h_{\alpha_i}=h^0_{\alpha_i}+ h^1_{\alpha_i}=h^0_{\beta_i}+h^0_{\theta} + h^1_{\beta_i}+h^1_{\theta}=0+1+0+0=1.$$ 
 
Therefore, 
 $$\rho(L(a_i,b_i),\alpha_i)=\tfrac{4}{a_i} 
\sum_{k=1}^{a_i-1}\cot(\tfrac{\pi k}{a_i})\cot(\tfrac{ \pi k (-b_i)}{a_i})  \sin^2(\tfrac {\pi k (-b_i)}{a_i}).
 $$

Thus Proposition \ref{APSindex}  implies that \begin{equation}
\label{R(e)} 
\Ind^+(E,\alpha)=\tfrac{2}{a}-3 + n+  \sum_{i=1}^n\tfrac{2}{a_i} 
\sum_{k=1}^{a_i-1}\cot(\tfrac{\pi k}{a_i})\cot(\tfrac{ \pi k b_i}{a_i})  \sin^2(\tfrac {\pi k b_i}{a_i}).
\end{equation} This agrees with  the formula of Fintushel-Stern \cite{FSpseudo} for their invariant $R(a_1, \cdots, a_n)$, an equivariant index computed using the Atiyah-Segal $G$-index theorem \cite{atiyah-segal} rather than the Atiyah-Patodi-Singer theorem as done here.      Indeed, their formula is obtained from this one by making the summation index substitution $r= k b_i$ in the inner sum, and using the fact that  $  a\sum_{j=1}^n\tfrac{b_j}{a_j}= 1$  to express the inverse of $b_i$ mod $a_i$ as $b_i^{-1}=\frac{a}{a_i}$. 
 The Neumann-Zagier formula (\ref{NZformula}) can then be applied to the summands in (\ref{R(e)}) to  explicitly compute $\Ind^+(E,\alpha)$ in terms of the Seifert invariants $a_1,\cdots, a_n$.   After some simplification one obtains
$$\Ind^+(E,\alpha)   = 2n-3 -2\sum_{i=1}^n K_i 
$$ 
where $K_i$ are integers satisfying
$0<b_i+K_ia_i<a_i$.

 \medskip

Assume  that  for the choice of $L$  whose Euler class   $e=e(L)$ is Poincar\'e dual to $F$,  $\Ind^+(L\oplus \epsilon,\alpha)>0$.   
Assume further that $\Sigma=\Sigma(a_1,\cdots, a_n)$ bounds a  positive definite manifold  $B$ with  $H_1(B;\Z/2)=0$. 

Extend $E$ over the union of $N=X\cup_{\Sigma}(-B)$ by taking the trivial bundle over $-B$. In this way we obtain a bundle over $N$ which we denote $E_N$.   The remark  after Proposition  \ref{APSindex}  shows that $\Ind^+(E_N,\alpha)=\Ind^+(L\oplus \epsilon,\alpha)>0$. 
Since   
${\tau}( L(a_i,b_i),\alpha)\ge \frac{1}{a_i}$, it follows that $\hat\tau(\partial N)\ge \min\{\frac{1}{a_i}\}>\tfrac{1}{a}=p_1(E_N,\alpha)$, and hence $\calm(E_N,\alpha)$ is compact.

Since $\Sigma$ is a homology sphere, $H^2(N)$ splits orthogonally with respect to the intersection form as $H^2(X)\oplus H^2(-B)=\spa(e)\oplus W$ with $W$ free abelian.  Thus Theorem \ref{CofL} implies that 
  $C(e)$ contains a single point, contradicting
Theorem \ref{furfinster}.  

This shows that   if $\Ind^+(E,\alpha)>0$,  then $ \Sigma(a_1,\cdots, a_n)$ cannot bound  any    positive definite manifold  $B$ with $H_1(B;\Z/2)=0$. This   recovers the original result of Fintushel-Stern \cite{FSpseudo} using cylindrical end moduli spaces in place of orbifold moduli spaces.

\medskip

Notice that the argument did not use any information about the $0$-framed component in Figure \ref{fig1}. In particular, we can replace this unknot by {\em any knot}  (and the punctured disk which makes up $F$ by a punctured Seifert surface for $K$) to get the following slight generalization of the Fintushel-Stern result:

\begin{theorem} \label{small}  Let $ a_1,\cdots , a_n$ be  a set of relatively prime integers which satisfy $R(a_1,\cdots, a_n)>0$.

 Let $K$ be any knot in $S^3$, and    let $\Sigma$ be the homology 3-sphere obtained from  $ \frac{a_i}{b_i}$ Dehn surgeries on $n$ parallel copies of the meridian of $K$ and $0$-surgery on $K$, where the $b_i$ are integers chosen so that $(a_1\cdots a_n)\sum_i\tfrac{b_i}{a_i}=1$.   Then $\Sigma$ does not bound a positive definite   4-manifold $B$ with  $H_1(B;\Z/2)=0$.\qed
\end{theorem} 
 To be explicit, for any 
 relatively prime positive integers $p,q$ and positive  integer $k$, the  hypothesis of Theorem \ref{small} holds with $(a_1, a_2, a_3)=(p,q, pqk-1)$. This is easily checked using the Neumann-Zagier formula.

\bigskip

We next  prove Theorem 1, a  generalization  to Seifert fibered rational homology spheres of the theorem of Furuta stated in the introduction.   The main difficulty not present in the original theorem concerns the count of reducibles: in contrast to the homology sphere case, $C(e)$  may have more than one class, and so the task will be to establish that the order of $C(e)$ is odd, in order to use Theorem \ref{furfinster} to reach a contradiction. Our purpose was not to find the most general theorem, but rather to illustrate the  new complexities that arise  when leaving the comfortable domain of integer homology spheres.  

We mention the related results in the articles of Endo \cite{endo}, B. Yu \cite{yu}, and Mukawa \cite{mukawa}.    Endo considers 2-fold branched covers of a certain infinite family of pretzel knots.  These branched covers are Seifert fibered integer homology spheres, and Endo proves the knots are linearly independent in the smooth knot concordance group by using   Furuta's theorem to rule out the existence of certain $\Z/2$-homology cobordisms which would exist if the knots satisfied a linear relation. Yu  proves that certain Montesinos knots are not slice by considering  their 2-fold covers, which are Seifert fibered rational homology spheres.   Mukawa proves that certain Seifert fibered rational homology spheres  have infinite order in the rational homology cobordism group.  The emphasis in Mukawa's article is on removing the $\Z/2$ constraints on the homology of the bounding 4-manifolds.  
In a forthcoming article, \cite{HK1}, we will use this machinery to establish   that the positive clasped untwisted Whitehead doubles of the $(2, 2^k-1)$ torus knots are linearly independent in the smooth knot concordance group.

\medskip

 We begin with some generalities on Seifert fibered rational homology spheres. 
Any Seifert fibered 3-manifold which fibers over $S^2$ has the Kirby diagram  given   in Figure \ref{fig1} for some choices of  relatively prime pairs of integers $(a_i,b_i)$ with $a_i>0$.  
 Let $a=a_1\cdots a_n$ and $d=a\sum_i\tfrac{b_i}{a_i}$.

As before, let $X$ be the 4-manifold indicated in Figure  \ref{fig1} (together with 3-handles) so that the boundary of $X$ is the union of lens spaces $Q=\sqcup_i  L(a_i,b_i)$    and the Seifert fibered 3-manifold $S=S(0;(a_1,b_1),\cdots, (a_n,b_n))$.     Then $S$ is a  rational homology sphere if and only if   $d $ is non-zero; in fact, the  order of $H_1(S)$ is $|d|$.
Reversing  orientation changes the sign of $d$. Assume, then that $S$ is a rational homology sphere and  orient $S$ so that $d>0$.  

 There is an exact sequence
 \begin{equation}
\label{exseq}0\to H_2(X)\to H_2(X,Q) \xrightarrow{\partial } H_1(Q)\to H_1(X)\to 0.
\end{equation}
From the handle description one computes this to be 
 $$0\to \Z \to \Z \to \oplus_i \Z/a_i \to H_1(X)\to 0.$$

The surface $F$ constructed  from the punctured disk bounded by the 0-framed component and the core of the 0-handle has boundary meeting the $L(a_i,b_i)$ in circles $\mu_i$, and is disjoint from $S$.  The circles $\mu_i$ generate $H_1(L(a_i,b_i))$.  The class $[F,\partial F]$ generates $H_2(X,Q)=\Z$,     and  $\partial ([F,\partial F])=[\partial F]=\mu_1+\cdots+\mu_n\in H_1(Q)$.   In particular, $H_1(X)\cong \big(\oplus_i \Z/a_i\big) /\langle \mu_1+\cdots+\mu_n\rangle$. Our construction requires $H_1(X;\Z/2)=0$; this happens  if  and only if  at most one $a_i$ is even.

As before,  
$$F\cdot aF= - (a_1\cdots a_n)\sum_{i=1}^n\tfrac{b_i}{a_i}=-d.$$
 Taking $E=L\oplus \epsilon$ with $e(L)\in H^2(X)$ Poincar\'e dual to $F$, 
one calculates $\Ind^+(E,\alpha)$ just as in (\ref{R(e)}) and obtains
 \begin{equation}\label{eq3.2}
\Ind^+(E,\alpha)= \tfrac{2d}{a}-3 + n+  \sum_{i=1}^n\tfrac{2}{a_i} 
\sum_{k=1}^{a_i-1}\cot(\tfrac{\pi k}{a_i})\cot(\tfrac{ \pi k b_i}{a_i})  \sin^2(\tfrac {\pi k b_i}{a_i}) 
 =  2n-3 -2\sum_{i=1}^n K_i.
\end{equation}
where $K_i$ are integers satisfying
$0<b_i+K_ia_i<a_i$.  The second  equality follows from  (\ref{NZformula}).

Chern-Simons invariants of flat $SO(3)$ connections on Seifert fibered rational homology spheres can be computed just as for their integer homology sphere counterparts. For our purposes the  
the following result will suffice.

 \begin{lemma} \label{cslemma3.2} Let $S$ be the rational homology sphere $S=S(0;(a_1,b_1),\cdots, (a_n,b_n))$ with  $d=a\sum_i\tfrac{b_i}{a_i}$ odd and positive, where the $a_i>0$ are relatively prime and $a=a_1\cdots a_n$. 
 \begin{enumerate}
\item  If $\gamma$ is a flat  irreducible $SO(3)$ connection on the trivial bundle $E=S\times \R^3$, $\cs(S,\gamma) =\frac{q}{a}$  {\rm mod} $\Z$  for some integer $q$.

\item   If  $\gamma$ is a flat   reducible $SO(3)$ connection on the trivial bundle, then $\cs(S,\gamma)=\frac{4q}{d}$ {\rm mod} $\Z$ for some integer $q$.
 \end{enumerate}
\end{lemma}

\begin{proof}  The fundamental group  $\pi=\pi_1(S)$ has the presentation 
$$\pi=\langle x_1,\cdots, x_n, h\ | \ h\text{ central}, \ x_i^{a_i}h^{b_i}, x_1\cdots x_n\rangle.$$   
Since $d$ is odd, $H^1(S;\Z/2)=0$, and so any representation $\pi_1(S)\to O(2)$ takes values in $SO(2)$.  It follows that an irreducible $SO(3)$ representation necessarily takes $h$ to $1\in SO(3)$, since the centralizer of any non-abelian subgroup of $SO(3)$ not conjugate into  $O(2)$ is trivial. The fundamental group of  $X$ is the quotient of $\pi$ by $h$, and hence any irreducible representation of $\pi$ extends to $\pi_1(X)$.  Thus $X$ provides a flat $SO(3)$ cobordism from $S$ to a disjoint union of lens spaces.

The   Chern-Simons invariants of the lens space  $L(a,b)$ are fractions with denominator $a$ and since Chern-Simons invariants modulo $\Z$ are flat cobordism invariants,  the Chern-Simons invariant of an irreducible representation of $\pi$ is a sum of fractions with denominators $a_1,\cdots, a_n$. Thus every irreducible flat $SO(3)$ connection $\gamma$ on $S$ satisfies $\cs(S,\gamma) =\frac{q}{a}$ for some integer $q$.

A reducible $SO(3)$ representation factors through $SO(2)$, hence has image in $\Z/d\subset SO(2)$.  Integrating over a fundamental domain in the $d$-fold cover and using the fact that the pullback flat connection has trivial holonomy and hence has Chern-Simons invariant  a multiple of 4 shows that $\cs(S, \gamma)=\frac{4q}{d}$ for some integer $q$.\end{proof}

 \medskip
 
   Let $\Theta^3_{\Z/2}$ denote  the group of oriented $\Z/2$ homology spheres modulo $\Z/2$ homology bordism. 
 We  give the following variant of the result of Furuta, which corresponds to  the case of $d=1$. This theorem implies Theorem 1 of the introduction.

\begin{theorem}  \label{SFQHS} Suppose that $ p,q,d $ are pairwise relatively prime positive odd integers.  Let $n_1,n_2,\cdots $ be a strictly increasing sequence of even positive integers satisfying
 $$ \gcd(d,n_k)=1,$$ 
 $$ n_k> dn_i -d(d-1)/pq  \text{ for all } k>i,  $$ 
and
$$n_1>\frac{d}{pq}\max\{ 1+ \frac{d}{pq}, 1+ \frac{1}{p },1+ \frac{1}{q}\}.$$
For example, when $d>1$ one can take $n_k=d^{k+s}-1$ for some fixed $s\gg0$.

Then the rational homology spheres obtained by $-\frac{d}{n_k}$ surgeries   on the left-handed $(p,q)$ torus knot  $k\ge 1$  are linearly independent in the $\Z/2$ homology cobordism group $\Theta^3_{\Z/2}$, and hence generate an   infinitely generated free abelian subgroup.  
 \end{theorem} 
\begin{proof}  Choose integers $r,s$ so that $ps+rq=-1$. Given an even integer $n\ge n_1$ satisfying $(d,n)=1$,
denote by $S_n$ the Seifert fibered rational  homology sphere  
$$S_{n}=S(0; (p,r), (q,s), (pqn -d, n))$$
(the hypotheses imply that $pqn-d>n$).
This manifold is diffeomorphic to $-\frac{d}{n}$ surgery  on the $(p,-q)$ torus knot.  One way to see this is to observe from Figure \ref{fig1} that $S_n$ is obtained from the Seifert fibration of $S^3= S(0;(p,r),(q,s))$   by performing Dehn surgery on a regular fiber, which is the left handed $(p,q)$ torus knot. For details see   \cite{Moser}.

  Let $X_{n}$ denote the corresponding negative definite 4-manifold  described in Figure \ref{fig1}, with boundary $S_{n}\sqcup Q$, where $Q =L(p,r)\sqcup L(q,s) \sqcup L(pqn-d,n)$. 
With the constructions given above, $H_1(S_{n})$ is cyclic of order $d$ and  $H_1(X_{n})=0$.

We argue  by contradiction. Suppose   there exists an oriented  $\Z/2$ homology punctured 4-ball $B$ with boundary 
$$ \partial B= -\sum_{i=1}^{N} a_{n_i} S_{n_i} $$
with $a_{n_N}\ne 0$.  We may assume, by changing orientation, that $a_{n_N}> 0$.

Let $Z_{n_N}$ be a negative definite 4-manifold with boundary $S_{n_N}$. For example, a simply connected $Z_{n_N}$ can be constructed by   gluing negative definite manifolds to $X_{n_N}$ along the lens spaces.  Note that this is possible since the lens space $L(a,b)$,  equipped with either orientation, bounds a negative definite 4-manifold obtained by plumbing disk bundles over $S^2$ according to the continued fraction expansion of $\frac{a}{b}$ (or, for $-L(a,b)$, the expansion of $\frac{a}{a-b}$).
 
Thus, we construct a negative definite 4-manifold $N$ with boundary  $$\partial N= L(p,r)\sqcup  L(q,s)\sqcup  L(pqn_N-d, n_N)  
   \sqcup_{i=1}^{N-1}  a_{n_i} S_{n_i},$$ 
   as an identification space from $ B$, $X_{n_N}$, and $ a_{n_N} -1$ copies of $Z_{n_N}$.  Here, $N$ is obtained by gluing  $ B$ to $X_{n_N}$ along one  $S_{n_N}$, and gluing the copies of $Z_{n_N}$ to $ B$ along the remaining copies of $S_{n_N}$.

For the convenience of the reader we simplify notation and illustrate it in Figure \ref{fig2}.   Let \begin{itemize}
\item $Q= L(p,r)\sqcup  L(q,s)\sqcup  L(pqn_N-d, n_N)$
\item $S$ denote the first copy of $S_{n_N}$,
\item $X=X_{n_N}$,
\item $Z$ denote the disjoint union of $ a_{n_N} -1$ copies of $ Z_{n_N}$,
\item $T= \sqcup_{i=1}^{N-1} a_{n_i} S_{n_i} $.
\end{itemize} Thus $N=X\cup B\cup Z$, the boundary is given by $\partial N=Q\sqcup T$, and $S$ separates $N$ into $X$  and $B\cup Z$.

\begin{figure}
\psfrag{X}{$X$}
\psfrag{B}{$B$}
\psfrag{Q}{$Q$}
\psfrag{S}{$S$}
\psfrag{T}{  {$T$}  }
\psfrag{Z}{$Z$}
\begin{center}
 \includegraphics [height=190pt]{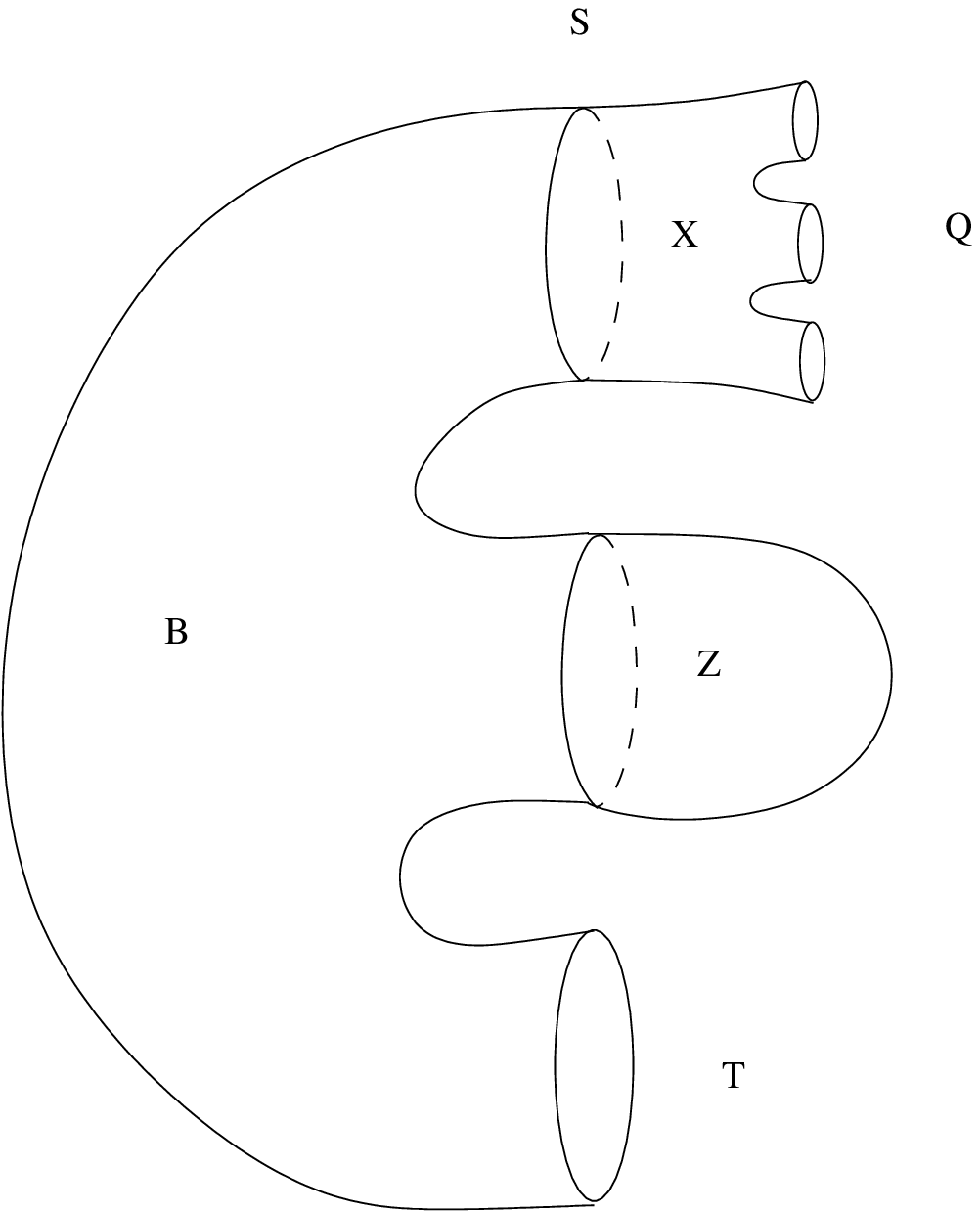}
\caption{\label{fig2} }
\end{center}
\end{figure}

Let $e\in H^2(N, T)$ denote the extension (by zero) of the class in $H^2(X,S)$ which is  Lefschetz dual to the generator $[F,\partial F]\in H_2(X,Q)=\Z$. Let $L\to N$ denote the corresponding $SO(2)$ vector bundle with Euler class $e$, which is trivial over $B\cup Z$. Let $E=L\oplus \epsilon$, and $(E,\alpha)$ the corresponding adapted $SO(3)$ bundle.

  Then 
$$p_1(E,\alpha)=-F\cdot F=  \tfrac{1}{pq(pqn_N-d)}\big(rq(pqn_N-d)+ps(pqn_N-d)+pqn_N\big)= \frac{d}{pq(pqn_N-d)}.$$
Suppose that  $K_1$ and $ K_2$ are integers satisfying $0<r+K_1p<p$ and $0<s+K_2 q<q$.
Then 
$$-1=ps+qr=p(s+K_2q)+q(r+K_1p)-pq(K_1+K_2)\leq p(q-1)+q(p-1)- pq(K_1+K_2)$$
so that 
  $$p+q-1\leq pq(2-(K_1+K_2))$$
 and hence $K_1+K_2\leq 1$.
Similarly 
$$-1\ge p+q-pq(K_1+K_2), $$ 
which implies $K_1+K_2\ge 1$. Thus $K_1+K_2=1$. 

Taking $(a_1,b_1)=(p,r), (a_2,b_2)=(q,s), (a_3,b_3)=(pqn_N-d,n_N)$  and $a=a_1a_2a_3$ in (\ref{eq3.2}) we obtain (since  $0<n_N<pqn_N-d$   we   take $K_3=0$):   
\begin{equation}
\label{eqindexcalc}
\Ind^+(E,\alpha)=
\tfrac{2d}{a} +  \sum_{i=1}^3\tfrac{2}{a_i} 
\sum_{k=1}^{a_i-1}\cot(\tfrac{\pi k}{a_i})\cot(\tfrac{ \pi k b_i}{a_i})  \sin^2(\tfrac {\pi k b_i}{a_i}) =3-2(K_1+K_2)= 1.
\end{equation}

Each component $S_{n_i}$ of $T$ is a Seifert fibered rational homology sphere with singular fibers of multiplicities $p,q$ and $pqn_i-d$.  The bundle $E$ is trivial over $T$.  Lemma \ref{cslemma3.2} implies that each irreducible flat connection  on the trivial $SO(3)$ bundle  over $S_{n_i}$ has Chern-Simons invariant a rational number with denominator 
$ {pq(pqn_i-d)} $  and reducible flat connections have Chern-Simons invariant a rational number with denominator $d$.
The (relative) Chern-Simons invariants of the lens spaces that make up $Q$ are  fractions  with denominators $p,q,$ and $ pqn_N-d  $ respectively. Lemma \ref{mod1mod4} then implies that
$$ \min\{ \tfrac{1}{pq(pqn_i-d)}, \tfrac{1}{d}, \tfrac{1}{p}, \tfrac{1}{q}, \tfrac{1}{pqn_N-d}\}\leq \hat\tau(Q\sqcup T,\alpha).$$

The hypothesis   $n_N\ge n_1>\frac{d}{pq}(1+ \frac{d}{pq})$ implies that  
 $$ \tfrac{d}{pq(pqn_N-d)}<\tfrac{1}{d}.$$ Similarly, $n_N >\frac{d}{pq}(1+ \frac{1}{p})$ implies 
 $$ \tfrac{d}{pq(pqn_N-d)}<\tfrac{1}{q}$$ and $n_N >\frac{d}{pq}(1+ \frac{1}{q})$ implies 
 $$ \tfrac{d}{pq(pqn_N-d)}<\tfrac{1}{p}.$$ Finally, the hypothesis $ n_N> dn_i -d(d-1)/pq $ implies that 
$$ \tfrac{d}{pq(pqn_N-d)}<\tfrac{1}{pq(pqn_i-d)}.$$
Taking the minimum of the right side for $i=1,2,\cdots, N-1$ we conclude
 \begin{equation}\label{eqn3.5}
 \tfrac{d}{pq(pqn_N-d)}=p_1(E,\alpha)< \hat\tau(Q\sqcup T,\alpha).
\end{equation}

 Since the path components of $Q$ are lens spaces, the restriction of  $\alpha$ to each component of $Q$ is non-degenerate. Since $T$ is a union of rational homology spheres and $\alpha$ restricts to the trivial connection on $T$, the restriction of  $\alpha$ to each component of $T$ is non-degenerate. 
With Equations (\ref{eqindexcalc}) and (\ref{eqn3.5}), this shows that the hypotheses of Theorem \ref{furfinster} hold. 
 
We now turn to the count of singular points (i.e. boundary points) of $\calm(E,\alpha)$. Since each  component of $T$ has first homology  $\Z/d$ with $d$  odd and  since $p,q,  pqn_N-d$ are odd, every component of $\partial N$ is a $\Z/2$ homology sphere.  Theorem \ref{CofL}  then implies that 
 $\calm(E,\alpha)$ has $|C(e)|$ boundary components, where
 $$C(e)=\{ e'\in H^2(N) \ | \ e'\cdot e'=-\tfrac{d}{a}, e'\equiv e \text{ mod } 2, e'|_T=0, e'|_{Q_i} =\pm m_i, i=1,2,3\}/\scriptstyle{\pm1}.$$
Here,  $Q_i$ are the three lens space components of $Q$ and $m_i\in H^2(Q_i)$ is the Poincar\'e dual to the meridional generator of $H_1(Q_i)$.

Since $p,q$, and $d$ are pairwise relatively prime,   $H_1(Q)=\Z/p\oplus \Z/q\oplus\Z/(pqn_N-d)=\Z/a$,  generated by $m_1+m_2+m_3$.  Suppose that $e'\in C(e)$.   Choose the unique integer $k$  with $0\leq k<a$ so that the restriction of $e'$ to $H^2(Q)$ equals $k(m_1+m_2+m_3)$.   Since $e'\in C(e)$, $e'$ vanishes on $T$. Note that the sign of $e'$ is ambiguously defined.  It can be uniquely specified by requiring $e'$ to restrict to $m_1$ in $H^2(L(p,r))$ (since $p$ is odd), and so we assume this. This implies that $k\equiv 1$ mod $p$.  Since $e'\in C(e)$, $k\equiv \pm 1$ mod $q$ and $k\equiv \pm 1$ mod $(pqn_N-d)$

The class $de'$ vanishes on $S$, and hence can be expressed  as $de'=e_1+ e_2$ where $e_1\in H^2(B\cup Z, \partial(B\cup Z))$ and $e_2\in H^2(X,S)$.   
Since $B\cup Z$ is negative definite, $e_1\cdot e_1=-\ell_1$ for some non-negative integer $\ell_1$.  The restriction of $e_2$ to $H^2(Q)$ equals $dk(m_1+m_2+m_3)$, and since 
$H^2(X,S)\cong \Z\langle e\rangle$, we have 
 $e_2=(dk+a\ell_2)e \text{ for some } \ell_2\in \Z,
$ and so $e_2\cdot e_2=-(dk+a\ell_2)^2d/a$.  Note $e_1\cdot e_2=0$.
 
 Then 
 $$-\frac{d}{a}=e'\cdot e'=\frac{1}{d^2}(e_1+e_2)\cdot (e_1+e_2)=-\frac{\ell_1}{d^2}-
 \frac{(dk+a\ell_2)^2}{da}.
 $$
We rewrite this as
 $d^3=\ell_1a+(dk+a\ell_2)^2 d.$ 
Since $d$ and $a$ are relatively prime, $\ell_1=d\ell_3$ for some integer $\ell_3\ge 0$, and so
 $$d^2=\ell_3a+(dk+a\ell_2)^2 .$$ 
The fact that $\tfrac{d}{pq(pqn_N-d)}<\tfrac{1}{d}$ implies $a>d^2$, from which it follows that $\ell_3=0$. We therefore have $d=\pm(dk+a\ell_2)$. Reducing this last equation mod $p$ yields $d\equiv \pm d$ mod $p$.  Since $p$ is odd  and $(d,p)=1$, it follows that 
 $d= dk+a\ell_2 .$  Reducing this equation mod $q$  and mod $pqn_N-d$ shows that $k\equiv 1$ mod $q$ and $k\equiv1$ mod $(pqn_N-d) $, and so $k=1$.  This implies that $\ell_2=0$, and that  the restriction of $e'$ to $Q$ {\em equals} the restriction of $e$ to $Q$. 
 
 Hence $e'-e$ vanishes on $\partial N$ and
 $$(e'-e)\cdot (e'-e) =- \tfrac{2d}{a}-2 e'\cdot e=- \tfrac{2d}{a}-\tfrac{2}{d}(de')\cdot e
 =- \tfrac{2d}{a}-\tfrac{2}{d} e_2 \cdot e=- \tfrac{2d}{a}-\tfrac{2}{d}(de)\cdot e=0.$$
Thus $e'-e\in H^2(N,\partial N)$ is a torsion class. Since $H^2(N,\partial N)$ has only odd torsion, it follows that $e'=e+2x$ for a torsion class $x\in H^2(N,\partial N)$.  

Conversely, every class of the form $e'+2x$ for $x\in H^2(N,\partial N)$  torsion is in $C(e)$, and hence
we have established a bijection between $C(e)$ and Torsion($H^2(N,\partial N)$). In particular, $C(e)$ is odd.

This contradicts the conclusion of Theorem \ref{furfinster}, and hence  the putative manifold $B$ cannot exist.
  \end{proof}
 
For an explicit example of Theorem \ref{SFQHS}, consider  $(p,q,d)=(3,5,7)$.  Then   the  homology lens spaces obtained  from $-\tfrac{7}{7^k-1}$ surgery on the $(3,5)$ torus knot for $k=1,2,\cdots\infty$ are linearly independent in $\Theta^3_{\Z/2}$.

\section{Example: non Seifert fibered homology 3-spheres}\label{examples}

Using various techniques of the type discussed in Section \ref{rhoinvts} and topological constructions, one can greatly extend the range of applications of our method.  For example, in \cite{HK1} we   prove that the untwisted Whitehead doubles of the $(2, 2^k-1)$ torus knots  are linearly independent in the smooth concordance group. 
We content ourselves for now with one further application, this time in a context which is homologically identical but geometrically distinct from the case of Seifert fibered homology spheres.   The motivation here is to illustrate some of the new issues which arise when one moves beyond the study of Seifert fibered 3-manifolds.

 \medskip

Suppose that $L=K_1\sqcup\cdots\sqcup K_n\subset S^3$ is a split link and $K_0\subset S^3\setminus L$ an additional component which links each $K_i$ once.  Perform Dehn surgery on $L\sqcup K_0$ with surgery coefficient $0$ on  $K_0$  and $\frac{a_i}{b_i}$ on $K_i$, where $(a_1\cdots a_n)\sum\frac{b_i}{a_i}=1$.  This configuration defines a negative definite 4-manifold $X$, whose boundary is a union of homology lens spaces $ Y_i=-S^3_{ {a_i}/{b_i}}(K_i)$ and a homology sphere $Y_0=\Sigma$ (the construction is homologically the same as the construction in Section \ref{FSexample}).   A Seifert surface for the 0-framed component can be chosen that meets each $K_i$ transversely once.  Let $F$  denote the corresponding punctured Seifert surface together with the core of the 2-handle, pushed slightly into the interior of $X$.  Let $e\in H^2(X;\Z)$ be Poincar\'e dual to the relative homology class represented by $F$.   
 
  The analysis of Section \ref{FSexample}  can be extended to this setting. 
Indeed, the homological data is identical to that of the Fintushel-Stern truncated mapping cylinder, and in the same way determines an adapted bundle $(E,\alpha)=(L_e\oplus \epsilon, \beta\oplus \theta)$. Therefore the part of the data that depends only on the homology, namely $C(e)$, $p_1(E,\alpha)= -e\cdot e$,  and the holonomy of the flat connection $\beta_i$ on $Y_i$, is the same as for the corresponding Seifert fibered example.  

However, some of the invariants that appear in the formula for $\Ind^+(E,\alpha)$ and $\hat{\tau}(Y,\alpha)$ depend on more than the homological data.  Since the $Y_i$ are no longer lens spaces,   to apply Theorem \ref{furfinster} one  one needs to check non-degeneracy of the $\alpha_i$ on $Y_i$.  One also must compute $\rho$ invariants of  the reducible flat connections $\alpha_i$ on $Y_i$.   Finally, one needs to estimate Chern-Simons invariants of flat connections on $\Sigma$ and   $Y_i, i\ge 1$.

The flat connection $\alpha_0$ is trivial on the homology sphere $Y_0$ and hence it is non-degenerate.  However, the flat connections $\alpha_i=\beta_i\oplus \theta$ for $i\ge 1$ need not be  non-degenerate for arbitrary choices of $K_i$ and $\frac{a_i}{b_i}$.     In fact $\alpha_i$ is non-degenerate if and only if $\Delta_{K_i}(\exp(2\pi i b_i/a_i))\ne 0$  \cite{K}, where $\Delta_K(t)$ denotes the Alexander polynomial of $K$.  
 
 \medskip

 We show how this works with one example, illustrated in Figure \ref{fig3}. In this figure $T$ can be  any 3-stranded tangle so that the resulting $K_0$ is connected.  The components $K_1$ and $K_3$ are unknotted and $K_2$ is the figure 8 knot.  By adding two 3-handles one obtains a negative definite 4-manifold $X$ with boundary the union of the homology sphere$Y_0=\Sigma$ given by this surgery diagram and the disjoint union of $Y_1=L(2,1), Y_3=L(11,-2)$ and the homology lens space $Y_2=-S^3_{-3/1}(K_2)=S^3_{3/1}(K_2)$, where the latter equality follows from the fact that $K_2$ is the amphicheiral figure 8 knot.  
 
 \begin{figure}
\psfrag{a}{$ \frac{2}{1}$}
\psfrag{b}{$-\frac{3}{1}$}
\psfrag{c}{$-\frac{11}{2}$}
\psfrag{K}{$T$}
\psfrag{d}{$0$}
\begin{center}
 \includegraphics [height=150pt]{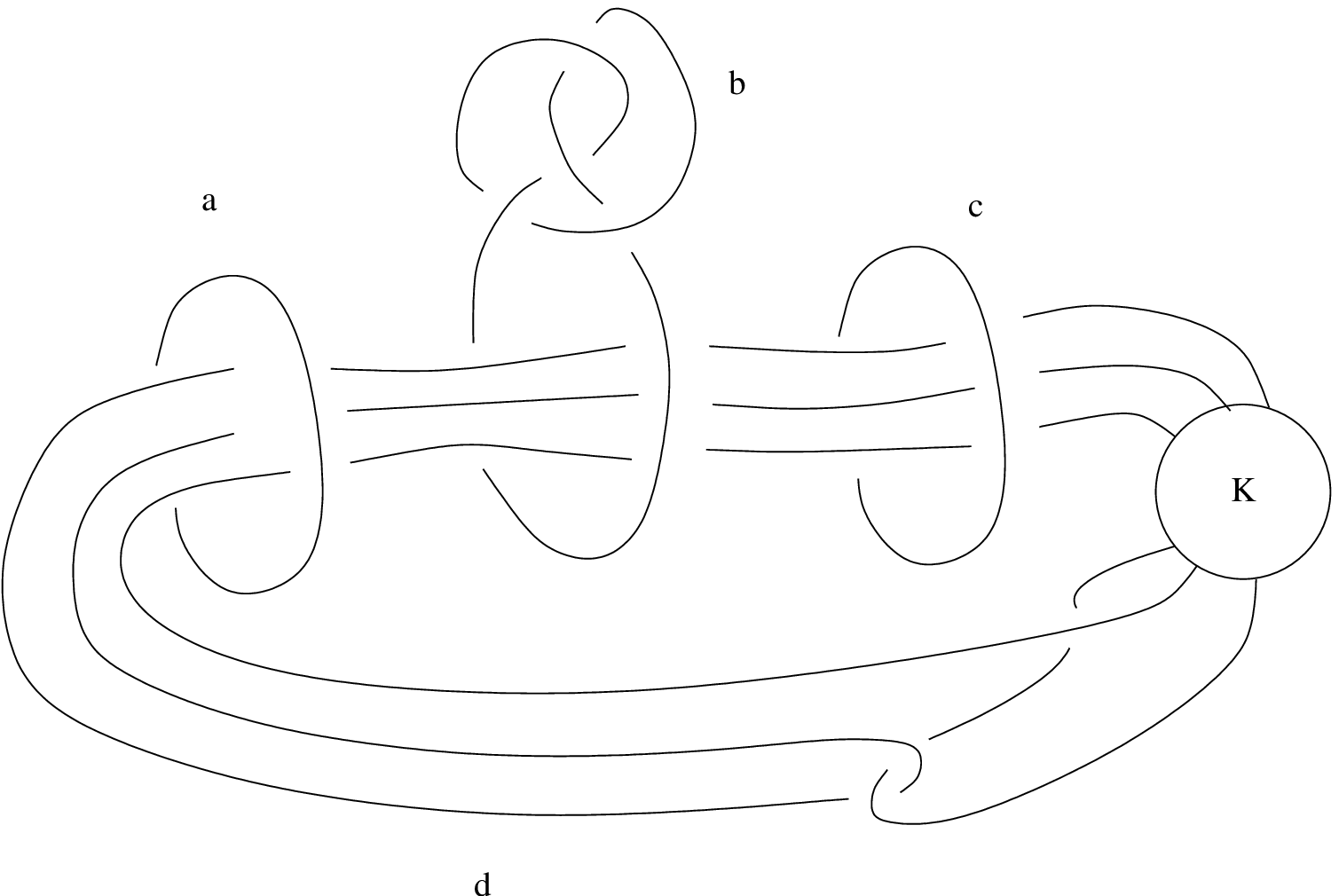}
\caption{\label{fig3} }
\end{center}
\end{figure}

 Since the homological data in this example is the same as that for $\Sigma(2,3,11)$, we see that $C(e)$ consists of a single point, $p_1(E,\alpha)= \frac{1}{66}$, and the holonomy of the flat connection $\beta_i$ takes $\mu_i$ to $\exp(2\pi i (-b_i)/a_i)$ for $i=1,2,3$, and is trivial for $i=0$. 

The Alexander polynomial of the figure 8 knot (resp. unknot) has no roots on the unit circle and so $\alpha_2$ (resp. $\alpha_1,\alpha_3$) is non-degenerate. 
Also,  $h_{\alpha_i}=1$, and so
applying Proposition \ref{APSindex} we obtain
$$\Ind^+(E,\alpha)= \tfrac{2}{66}  -\tfrac{1}{2}\sum_{k=1}^3 \rho(Y_i,\alpha_i).
$$
The $\rho$ invariants of $Y_1=L(2,1)$ (resp. $Y_3=L(11,-1)$) with respect to $\alpha_1$ (resp. $\alpha_3)$ were computed above. We compute  $\rho(Y_2,\alpha_2)$     using the Atiyah-Patodi-Singer theorem, as described in 
Section \ref{rhoinvts}.   

For any knot $K\subset S^3$ there exists  a cobordism (rel boundary)  $Z_K$  from $S^3\setminus n(K)$ to $S^3\setminus n(U)$ (here $U$ is the unknot), together with a homomorphism $H_1(Z_K)\to\Z$ which extends the abelianization map on the knot complement.  One construction of such a $Z_K$ is to push a Seifert surface for $K$ into $D^4$ and perform surgery along circles to turn the Seifert surface into a $2$--disk. The resulting 4-manifold contains a $2$--disk whose complement  is $Z_K$.    For more details on this  construction, we refer the reader to \cite{KKR}. An alternate construction for such a $Z_K$ can be found in \cite{HK1}. 

One can then glue $I\times (S^1\times D^2) $ to $Z_K$ to obtain a cobordism $W_K$ from $-\frac{a}{b}$ surgery on $K$ to $L(a,b)$  over which $\beta$ (and hence $\alpha$) extends.  Since $\alpha=\beta\oplus 1$ which complexifies to $\beta_{2 _\C}\oplus \bar\beta_{2 _\C}\oplus 1$, we conclude that 
$$\rho(Y_2,\alpha_2)=\rho(L(3,1),\alpha_2) +      \Sign_{\beta_{2 _\C}}(W_{K_2}) + 
\Sign_{\bar\beta_{2 _\C}}(W_{K_2})- 2 \Sign(W_{K_2}).$$
The quantity $\Sign_{\beta_{2 _\C}}(W_{K_2}) -\Sign(W_{K_2})$ equals the Levine-Tristram signature (\cite{le,tr}) of the knot  $K_i$ at the $U(1)$-representation $\beta_{2_\C}$ which sends the meridian of $K_i$ to $e^{2\pi i (-b_2)/a_2}$. This can be proven using Wall non-additivity \cite{wall}, and the details are carried out in \cite{KKR}.  Since $K_2$ is the figure 8 knot,  all the Levine-Tristram signatures of $K_2$  are zero. This follows, for instance, from the fact that the Alexander polynomial of the  figure 8 knot has no roots on the unit circle.   Hence $\rho(Y_2,\alpha_2)=\rho(L(3,1),\alpha_2)$ and so   
 \begin{equation}
\label{R(e)fake}
\Ind^+(E,\alpha)= R(2,3,11)=1.
\end{equation}

To apply Theorem \ref{furfinster}  requires an understanding of $\hat{\tau}(Y,\alpha)$.   This is the point where the method requires a deeper analysis of flat $SO(3)$ connections on $Y_i$.  In the Seifert fibered case each $K_i$ is an unknot, and $Y_i$ is a lens space.  Thus all $SO(3)$ representations are abelian, hence reducible, and their Chern-Simons invariants are easy to compute.

For the example of Figure \ref{fig3}, 
\begin{equation}
\label{eqn6.2}
p_1(E,\alpha)=\tfrac{1}{66},\ {\tau}(Y_1,\alpha_1)\ge \tfrac{4}{2} ,\text{ and }{\tau}(Y_3,\alpha_3)\ge \tfrac{4}{11},
\end{equation}
  since $Y_1=L(2,1)$ and $Y_3=L(11,-2)$.  The Chern-Simons is a flat cobordism invariant modulo $\Z$, and so, using the cobordism described above when computing $\rho$ invariants of $Y_2$, we conclude that the Chern-Simons invariants of the  reducible flat connections over $Y_2$ agree modulo $\Z$ with those of $L(3,1)$. Hence  $$\{\cs(Y_2,\alpha_2,\gamma)\}\ge  \tfrac{1}{3},$$where $\{x\}$ denotes the fractional part of of a real number $x$.

However, $Y_2$ admits irreducible flat $SO(3)$ connections, whose Chern-Simons invariants may contribute to $\hat{\tau}$.    A method for computing Chern-Simons invariants of surgeries on the figure 8 knot is described in \cite{KK1}. In particular, \cite{KK1} calculates Chern-Simons invariants of $-3$ surgery on the figure 8 knot. There are two gauge equivalence classes of irreducible $SO(3)$ flat connections 
on $Y_2$, and  both have Chern-Simons invariants rational numbers with denominator $24$.   Hence for $\gamma$ an irreducible flat connection on $Y_2$, $ \cs(Y_2, \alpha_2,\gamma) $ is   a difference of   fractions with denominators $3$ and    $24$. Therefore $$\{\cs(Y_2, \alpha_2,\gamma))\}\ge\tfrac{1}{24} .$$ 
Lemma \ref{mod1mod4} implies  that $\tau(Y_2,\alpha_2)\ge \frac{1}{24}$.  Together with Equation (\ref{eqn6.2}) this shows that   $\hat \tau(Y,\alpha)>p_1(E,\alpha)$.

 Theorem \ref{furfinster} then implies that   $Y_0$ does not bound a positive definite 4-manifold.   Moreover, a simple adaptation of the argument shows that $Y_0$ is of infinite order in $\Theta^3_{\Z/2}$.  For this case, one must obstruct the existence of a punctured $\Z/2$ homology ball whose boundary consists of a disjoint union of manifolds,  each of which are orientation-preserving diffeomorphic to $-Y_0$.  For this, we cap off all of the $-Y_0$ with the negative definite $4$--manifold $X$ from the previous argument.  Consider the adapted bundle $(E,\alpha)$ over the resulting manifold which is the stabilization of the $SO(2)$ bundle whose Euler class is Poincar{\'e} dual to the generator of $H^2(X)$ (for one copy of $X$).    The estimates of the Chern-Simons invariant for flat connections on lens spaces and $Y_2$ given above show that the punctured $\Z/2$-homology ball does not exist, proving that $Y_0$ has infinite order.  Note that to apply Theorem \ref{furfinster} here one must bound the relative Chern-Simons invariants $\cs(Y_i, \theta_i,\gamma))>\frac{1}{66}$, where $\theta_i$ is the trivial connection on the trivial bundle (corresponding to the remaining copies of $X$ on which the $SO(2)$  bundle restricts trivially).  This estimate, however, follows exactly as in the discussion giving the bounds $\cs(Y_i, \alpha_i,\gamma))>\frac{1}{66}$.  
 
\bibliographystyle{amsplain}

\end{document}